\documentclass{article}

\usepackage{amsmath,amssymb,amsthm,mathrsfs}
\usepackage[overload]{empheq}
\usepackage{graphicx}
\usepackage[colorlinks=true]{hyperref}
\usepackage{pdfsync}
\usepackage{tikz}
\usetikzlibrary{decorations.pathreplacing}
\usetikzlibrary{patterns}

\newtheorem{theorem}{Theorem}
\newtheorem{proposition}{Proposition}
\newtheorem{lemma}{Lemma}

\theoremstyle{definition}
\theoremstyle{definition}
\theoremstyle{definition}\newtheorem{remark}{Remark}

\newcommand{\fonction}[5]{\begin{array}[t]{lrcl}#1 :&#2 &\longrightarrow &#3\\&#4& \longmapsto &#5 \end{array}}
 
\newcommand{\vertiii}[1]{{\left\vert\kern-0.15ex\left\vert\kern-0.15ex\left\vert #1 
    \right\vert\kern-0.15ex\right\vert\kern-0.15ex\right\vert}}  

\DeclareMathOperator*{\argmin}{arg\,min}

\def\bs{\backslash}

\def\t{\tau}

\def\R{\mathbb{R}}
\def\N{\mathbb{N}}

\def\C{\mathrm{C}}
\def\AC{\mathrm{AC}}
\def\BF{\mathrm{BF}}
\def\BV{\mathrm{BV}}
\def\NBV{\mathrm{NBV}}

\def\V{\mathrm{V}}

\def\L{\mathrm{L}}
\def\E{\mathrm{E}}
\def\S{\mathrm{S}}
\def\B{\mathrm{B}}
\def\S{\mathrm{S}}
\def\K{\mathrm{K}}
\def\Q{\mathrm{Q}}
\def\RR{\mathrm{R}}

\def\UU{\mathscr{U}}

\title{Note on Pontryagin maximum principle with running state constraints and smooth dynamics - Proof based on the Ekeland variational principle}

\author{Lo\"ic Bourdin. University of Limoges, France.}

\begin{document}

\maketitle

\begin{abstract}
In this note our aim is to give a proof of the Pontryagin maximum principle for a general optimal control problem with running state constraints and smooth dynamics. Our proof is based on the classical Ekeland variational principle.

The main result (and its proof) of this note are not new and are already well-known in the literature. The aim of the author is only to provide a complete and detailed proof of this classical theorem in the case of smooth dynamics. If you have any remarks or questions, do not hesitate to contact the author at \textit{loic.bourdin@unilim.fr}.
\end{abstract}

\tableofcontents

\section{Main result}
We first introduce some notations available throughout the paper. Let $T>0$ be fixed. For every $n \in \N^*$ and every $1 \leq r \leq +\infty$, we denote by 
\begin{itemize}
\item $\BF_n := \BF ([0,T],\R^n)$ the classical space of bounded functions endowed with the classical uniform norm $\Vert \cdot \Vert_\infty$;
\item $\C_n := \C ([0,T],\R^n)$ the classical space of continuous functions endowed with $\Vert \cdot \Vert_\infty$;
\item $\AC_n := \AC ([0,T],\R^n)$ the classical space of absolutely continuous functions;
\item $\BV_n := \BV ([0,T],\R^n)$ the classical space of functions with bounded variations endowed with $\Vert \cdot \Vert_{\BV_n}$ (see Appendix~\ref{annexeBV} for some recalls);
\item $\L^r_n : = \L^r ([0,T],\R^n)$ the classical Lebesgue space of $r$-integrable functions endowed with its usual norm $\Vert \cdot \Vert_{\L^r_n}$.
\end{itemize}
In the whole paper, when no confusion is possible, we remove the subscript $n$ and we just denote by $\BF$, $\C$, $\AC$, $\BV$ or $\L^r$. \\

We denote by $\BF_1^+ := \BF([0,T],\R^+)$ and $\C_1^+ := \C([0,T],\R^+)$ where $\R^+ = [0,+\infty)$. \\

Then, $\eta \in \BV_n$ is said to be \textit{normalized} if $\eta (0) = 0 $ and $\eta$ is left-continuous on $(0,T)$. The subspace of normalized functions with bounded variations will be denoted by $\NBV_n$. \\

Finally, the classical Lebesgue measure on $[0,T]$ will be denoted by $\lambda$.

\subsection{A state constrained optimal control problem}
Let $m$, $n$ and $j \in \N^*$ be fixed. In this paper we consider the optimal control problem~\eqref{theproblem} given by
\begin{equation}\label{theproblem}\tag{OCP}
\begin{array}{ll}
\text{minimize} & \Psi ( q(T) ), \\
& \\
\text{subject to} & q \in \AC_n, \; u \in \L^\infty_m, \\[6pt]
& \dot q(t) = f ( q(t),u(t),t ), \quad \text{a.e. $t \in [0,T]$,}  \\[6pt]
& q(0)=q_0 , \\[6pt]
& u(t) \in \Omega, \quad \text{a.e. $t \in [0,T]$,}  \\[6pt]
& G_i(q(t),t) \leq 0, \quad \forall t \in [0,T], \quad \forall i=1,\ldots,j , \\[1pt]
\end{array}
\end{equation}
where $\Psi : \R^n \to \R$ is of class $\C^1$, where $f : \R^n \times \R^m \times [0,T] \to \R^n$ is continuous and of class $\C^1$ in its two first variables, where $G =(G_i)_{i=1,\ldots,j} : \R^n \times [0,T] \to \R^j$ is continuous and of class $\C^1$ in its first variable, and where $q_0 \in \R^n$ is fixed and $\Omega \subset \R^m$ is a nonempty closed subset. \\

Since $\Psi$, $f$ and $G$ are all regular, Problem~\eqref{theproblem} is said to be an optimal control problem \textit{with smooth dynamics}. The last constraint corresponds to \textit{running state constraints}. \\

We now introduce the Hamiltonian $H : \R^n \times \R^m \times \R^n \times [0,T] \to \R$ associated to Problem~\eqref{theproblem} defined by
$$ H(q,u,p,t) := \langle p , f(q,u,t) \rangle_{\R^n \times \R^n}. $$

\subsection{Pontryagin maximum principle}

Our main result in this note is given by the following theorem.

\begin{theorem}[Pontryagin maximum principle]\label{mainresult}
Let $(q^*,u^*) \in \AC_n \times \L^\infty_m$ be an optimal solution of Problem~\eqref{theproblem}. There exists a nontrivial couple $(\psi , \eta)$ where $\psi \geq 0$ and $\eta = (\eta_i)_{i=1,\ldots,j} \in \NBV_j$ such that
$$ u^*(t) \in \argmin_{v \in \Omega} H(q^*(t),v,p(t),t ) $$
for a.e. $t \in [0,T]$, where $p \in \BV_n$ is the unique global solution of the backward linear Cauchy-Stieltjes problem given by
$$ \left\lbrace \begin{array}{l}
-dp = \partial_1 H ( q^*,u^*,p,\cdot ) \; dt + \sum_{i=1}^j \partial_1 G_i (q^*,\cdot) \; d\eta_i , \quad \text{on $[0,T]$,}  \\[5pt]
p(T)= \psi \nabla \Psi ( q^*(T) ).
\end{array} \right. $$
In addition, it holds that
$$ \eta_i \text{ is monotically increasing on $[0,T]$} \quad \text{and} \quad \int_0^T G_i (q^*(\t),\t ) \; d\eta_i (\t) = 0, $$
for every $i=1,\ldots,n$.
\end{theorem}

We refer to Appendix~\ref{annexeBV} for some recalls about functions of bounded variations and to Appendix~\ref{appderniere} for details on linear Cauchy-Stieltjes problems.

\section{Proof}

This section is entirely devoted to the proof of Theorem~\ref{mainresult}. This proof is based on the classical Ekeland variational principle and is inspired from several references like \cite{bonnans,liyong}.

\subsection{Preliminaries}

Let $u \in \L^\infty$. In this preliminary section we focus on the forward (nonlinear) Cauchy problem~\eqref{eqCPU} given by
\begin{equation}\label{eqCPU}\tag{CP${}_u$}
\left\lbrace \begin{array}{l}
\dot q(t) = f ( q(t),u(t),t ), \quad \text{a.e. $t \in [0,T]$,}  \\[5pt]
q(0)=q_0.
\end{array} \right. 
\end{equation}
A couple $(q,I)$ is said to be a \textit{(local) solution} of \eqref{eqCPU} if
\begin{enumerate}
\item $I \subset [0,T]$ is an interval with nonempty interior such that $\min I = 0$;
\item $q : I \to \R^n$ is absolutely continuous on $I$ and $q$ satisfies
$$ \left\lbrace \begin{array}{l}
\dot q(t) = f ( q(t),u(t),t ), \quad \text{a.e. $t \in I$,}  \\[5pt]
q(0)=q_0,
\end{array} \right. $$
or equivalently, $q : I \to \R^n$ is continuous on $I$ and $q$ satisfies
$$ q(t) = q_0 + \int_0^t f(q(\t),u(\t),\t) \; d\t , $$
for every $t \in I$.
\end{enumerate}
The couple $(q,I)$ is said to be a \textit{global solution} of \eqref{eqCPU} if $I = [0,T]$. \\

Let $(q,I)$ and $(q',I')$ be two local solutions of \eqref{eqCPU}. We say that $(q,I)$ is an \textit{extension} of $(q',I')$ if $I' \subset I$ and $q (t) = q' (t)$ for every $t \in I'$. We say that $(q,I)$ is a \textit{maximal solution} of \eqref{eqCPU} if it extends all other local solutions of \eqref{eqCPU}.

\subsubsection{Some Cauchy-Lipschitz and continuous dependence results}

Recall the two following classical Cauchy-Lipschitz (or Picard-Lindel\"of) results.

\begin{lemma}\label{lemcauchylipsch}
For every $u \in \L^\infty$, there exists a unique maximal solution of \eqref{eqCPU}. 
\end{lemma}

In the sequel we denote by $(q(\cdot,u),I(u))$ the maximal solution of \eqref{eqCPU} associated to $u \in \L^\infty$.

\begin{lemma}
Let $u \in \L^\infty$. If $(q(\cdot,u),I(u))$ is not global (that is, $T \notin I(u)$) then $I(u)$ is not closed and $q(\cdot,u)$ is unbounded on $I(u)$.
\end{lemma}

In the sequel we denote by $\UU \subset \L^\infty$ the set of controls $u \in \L^\infty$ such that $T \in I(u)$. A control $u \in \UU$ is usually said to be \textit{admissible}. For every $u \in \UU$ and every $R > \Vert u \Vert_{\L^\infty}$, we introduce
$$ \K_{u,R} := \{ (x,v,t) \in \R^{n} \times \R^m \times [0,T] \; \mid \; \Vert x-q(t,u) \Vert_{\R^n} \leq 1 \text{ and } \Vert v \Vert_{\R^m} \leq R \} .$$
From continuity of $q(\cdot,u)$ on $[0,T]$, $\K_{u,R}$ is a compact subset of $\R^{n} \times \R^m \times [0,T]$. As a consequence, $f$, $ \partial_1 f $ and $\partial_2 f $ are bounded on $\K_{u,R}$ by some $L_{u,R} \geq 0$ and it holds that
\begin{equation}\label{eqlipK}
\Vert f(x_2,v_2,t) - f(x_1,v_1,t) \Vert_{\R^n}  \leq L_{u,R} ( \Vert x_2 -x_1 \Vert_{\R^n}  + \Vert v_2 -v_1 \Vert_{\R^m}  ),
\end{equation}
for all $(x_1,v_1,t)$, $(x_2,v_2,t) \in \K_{u,R}$.

\begin{proposition}\label{propEUR}
Let $u \in \UU$. For every $R > \Vert u \Vert_{\L^\infty}$, there exists $\nu_{u,R}>0$ such that
$$ \E_{u,R} := \overline{\B}_{\L^\infty} (0,R) \cap \overline{\B}_{\L^1} (u,\nu_{u,R}) $$
is contained in $\UU$. Moreover, for every $u' \in \E_{u,R}$, it holds that $(q(\t,u'),u'(\t),\t) \in \K_{u,R}$ for a.e. $\t \in [0,T]$.
\end{proposition}

\begin{proof}
Let $R > \Vert u \Vert_{\L^\infty}$ and let $\nu_{u,R} >0$ be such that $ \nu_{u,R} L_{u,R} e^{T L_{u,R} } < 1$. Let $u' \in \E_{u,R}$. Our aim is to prove that $T \in I(u')$. By contradiction, let us assume that the set
$$ A := \{ t \in I(u') \; \mid \; \Vert q(t,u') - q(t,u) \Vert_{\R^n}  > 1 \}$$
is not empty and let $t_0 := \inf A$. From continuity, it holds that $\Vert q(t_0,u') - q(t_0,u) \Vert_{\R^n} \geq 1$. Moreover, one has $t_0 >0$ since $q(0,u') = q(0,u) =q_0$. Hence, $\Vert q(\t,u') - q(\t,u) \Vert_{\R^n} \leq 1$ for every $\t \in [0,t_0)$. Therefore $(q(\t,u'),u'(\t),\t)$ and $(q(\t,u),u(\t),\t)$ belong to $\K_{u,R}$ for a.e. $\t \in [0,t_0)$.
Since one has
$$
q(t,u') - q(t,u) = \int_0^t  f ( q(\t,u'),u'(\t),\t ) - f ( q(\t,u),u(\t),\t ) \; d\tau ,
$$
for every $t \in I(u')$, it follows from~\eqref{eqlipK} that
\begin{multline*}
\Vert q(t,u') - q(t,u) \Vert_{\R^n} \leq L_{u,R} \int_0^t \Vert  u'(\tau) - u(\tau) \Vert_{\R^m}  \; d \tau  + L_{u,R} \int_{0}^t \Vert   q(\tau,u') - q(\tau,u) \Vert_{\R^n} \; d \tau ,  
\end{multline*}
for every $t \in [0,t_0]$, which implies from the classical Gronwall lemma that
\begin{equation*}
\Vert q(t,u') - q(t,u) \Vert_{\R^n} \leq  L_{u,R} e^{T L_{u,R} } \Vert u'- u \Vert_{\L^1} \leq \nu_R L_{u,R} e^{T L_{u,R} } < 1 ,
\end{equation*}
for every $t \in [0,t_0]$. This raises a contradiction at $t=t_0$. Therefore $A$ is empty. We conclude that $q(\cdot,u')$ is bounded on $I(u')$, then $T \in I(u')$. Moreover, since $A$ is empty, we also conclude that $\Vert q(t,u') - q(t,u) \Vert_{\R^n}  \leq 1$ for every $t \in [0,T]$, and thus $(q(\t,u'),u'(\t),\t) \in \K_{u,R}$ for a.e. $\t \in [0,T]$.
\end{proof}

We conclude this section with the following continuous dependence result.

\begin{proposition}\label{prop30-1-1}
Let $u \in \UU$ and $R > \Vert u \Vert_{\L^\infty}$. The mapping
\begin{equation*}
\fonction{F_{u,R}}{(\E_{u,R},\Vert \cdot \Vert_{\L^1})}{(\C_n,\Vert \cdot \Vert_\infty)}{u'}{q(\cdot,u')}
\end{equation*}
is $C_{u,R}$-Lipschitz continuous for some $C_{u,R} \geq 0$. 
\end{proposition}

\begin{proof}
Let $u'$ and $u''$ be two elements of $\E_{u,R} \subset \UU$. We know that $(q(\t,u''),u''(\t),\t)$ and $(q(\t,u'),u'(\t),\t)$ are elements of $\K_{u,R}$ for a.e. $\t \in [0,T]$. Following the same arguments as in the previous proof, it follows that
\begin{equation*}
\Vert q(t,u'') - q(t,u') \Vert_{\R^n}  \leq  L_{u,R} e^{T L_{u,R}} \Vert u'' - u ' \Vert_{\L^1},
\end{equation*}
for every $t \in [0,T]$. The lemma follows with $C_{u,R} : = L_{u,R} e^{T L_{u,R}} \geq 0$.
\end{proof}

\subsubsection{Implicit spike variations and a differentiable dependence result}

Before introducing the concept of implicit spike variations, we first need to recall the following lemma (see \cite[Paragraph 3.2 p.143]{liyong}). The proof is recalled in Appendix~\ref{appunity}.

\begin{lemma}\label{lem123684}
Let $h \in \L^1_n$. Then, for all $\rho \in (0,1)$, there exists a measurable set $\Q_\rho \subset [0,T]$ such that $\lambda (\Q_\rho ) = \rho T$ and
$$ \sup_{t \in [0,T]} \left\Vert \int_0^t \left( 1 - \frac{1}{\rho} \mathbf{1}_{\Q_\rho} (s) \right) h(s) \; ds \right\Vert_{\R^n} \leq \rho. $$
Note that $\Q_\rho$ depends on $h$.
\end{lemma}

Let $u \in \UU$ and $u' \in \L^\infty$. For every $\rho \in [0,1)$, we introduce the so-called \textit{implicit spike variation} $u(\cdot,\rho)$ of $u$ associated to $u'$ as
$$ u(\t,\rho) := \left\lbrace \begin{array}{lcr}
u'(\t) & \text{if} & \t \in \Q_\rho, \\
u(\t) & \text{if} & \t \notin \Q_\rho,
\end{array} \right. $$
for a.e. $\t \in [0,T]$, where $\Q_\rho$ is defined in Lemma~\ref{lem123684} associated to $h_{u,u'} \in \L^\infty_n \subset \L^1_n$ defined by
$$ h_{u,u'} (\t) := f(q(\t,u),u'(\t),\t)-f(q(\t,u),u(\t),\t) , $$
for a.e. $\t \in [0,T]$.\footnote{For $\rho = 0$, we fix $\Q_\rho = \emptyset$.} \\

Finally, we introduce the so-called \textit{variation vector} $ w(\cdot,u,u')$ associated to $(u,u')$ as the unique maximal solution, which is moreover global (see Appendix~\ref{appcauchyproblemclassique}), of the forward linear Cauchy problem given by
$$ \left\lbrace \begin{array}{l}
\dot w(t) = \partial_1 f(q(t,u),u(t),t) \times w(t) + h_{u,u'} (t) , \quad \text{a.e. $t \in [0,T]$,}  \\[5pt]
w(0)=0.
\end{array} \right. $$
Let us prove the following differentiability dependence result.

\begin{proposition}\label{prop984357}
The mapping $F_{u,u'}$ defined by
\begin{equation*}
F_{u,u'}(\rho) := q(\cdot,u(\cdot,\rho)) \in \C_n,
\end{equation*}
for sufficiently small $\rho \geq 0$, is Fr\'echet-differentiable at $\rho = 0$, with $DF_{u,u'}(0) = w(\cdot,u,u')$.
\end{proposition}

\begin{proof}
Let $R := \max ( \Vert u \Vert_{\L^\infty}+1 , \Vert u' \Vert_{\L^\infty} )$. Since $\lambda (\Q_\rho) = \rho T$ (see Lemma~\ref{lem123684}), it holds that $\Vert u(\cdot,\rho) - u \Vert_{\L^1} \leq 2RT\rho$ for every $\rho \in [0,1)$. As a consequence, for sufficiently small $\rho \geq 0$, $u(\cdot , \rho) \in \E_{u,R} \subset \UU$ and then $F_{u,u'} (\rho)$ is well-defined. Moreover, it follows from Proposition~\ref{prop30-1-1} that $\Vert q(\cdot,u(\cdot,\rho)) - q(\cdot,u) \Vert_\infty \leq 2RT C_{u,R} \rho$, and consequently $q(\cdot,u(\cdot,\rho))$ uniformly converges on $[0,T]$ to $q(\cdot,u)$. \\

Let us assume by contradiction that $F_{u,u'}$ is not Fr\'echet-differentiable at $\rho = 0$ with $DF_{u,u'}(0) = w(\cdot,u,u')$. Then, there exists $\varepsilon > 0$ et $(\rho_k)_k$ a positive sequence such that $(\rho_k)_k$ tends to zero and such that
$$ \left\Vert \frac{F_{u,u'} (\rho_k) -F_{u,u'} (0) }{\rho_k}-w(\cdot,u,u') \right\Vert_\infty \geq \varepsilon $$ 
for all $k \in \N$. In this proof, for the ease of notations, we denote by $w := w(\cdot,u,u')$, $q := q(\cdot,u)$ and by $q_k := q(\cdot,u(\cdot,\rho_k))$, $u_k := u(\cdot,\rho_k)$ for every $k \in \N$. Since the sequence $(u_k)_k$ converges to $u$ in $\L^1$, we deduce from the (partial) converse of the classical Lebesgue dominated convergence theorem that there exists a subsequence (that we do not relabel) such that $(u_k)_k$ tends to $u$ a.e. on $[0,T]$. For every $k \in \N$ and every $t \in [0,T]$, we define $ z_k (t) :=  \frac{q_k(t) - q(t)}{\rho_k} - w(t)$. From our assumption, it holds that $\Vert z_k \Vert_\infty \geq \varepsilon$ for all $k \in \N$. On the other hand, we have
$$
z_k (t) = \int_0^t \dfrac{ f(q_k(\t),u_k(\t),\t) - f(q(\t),u(\t),\t) }{\rho_k} \\[3pt] 
- \partial_1 f(q(\t),u(\t),\t) \times w(\t) - h_{u,u'}(\t) \; d\t, 
$$
that is,
\begin{multline*}
z_k (t) = \int_0^t \dfrac{ f(q_k(\t),u_k(\t),\t) - f(q(\t),u_k(\t),\t) }{\rho_k} - \partial_1 f(q(\t),u(\t),\t) \times w(\t) \\[3pt] 
+ \dfrac{ f(q(\t),u_k(\t),\t) - f(q(\t),u(\t),\t) }{\rho_k} - h_{u,u'}(\t) \; d\t,
\end{multline*}
for every $t \in [0,T]$. From the classical Taylor formula with integral rest, we obtain that
\begin{multline*}
z_k (t) = \int_0^t \partial_1 f(q(\t),u(\t),\t) \times z_k (\t) \; d\t + \int_0^t \left( \frac{1}{\rho_k} \mathbf{1}_{\Q_{\rho_k}} (\t) - 1 \right) h_{u,u'}(\t) \; d\t \\[3pt]
+ \int_0^t \left[ \int_0^1 \partial_1 f(q(\t)+ \theta (q_k(\t)-q(\t)),u_k(\t),\t) \; d\theta - \partial_1 f(q(\t),u(\t),\t) \right] \frac{q_k (\t) - q(\t)}{\rho_k} \; d\t,
\end{multline*}
for every $t \in [0,T]$. Hence, from Lemma~\ref{lem123684}, we get that
$$ \Vert z_k (t) \Vert_{\R^n} \leq \rho_k + 2 RT C_{u,R} \kappa_k + L_{u,R} \int_0^t \Vert z_k (\t) \Vert_{\R^n} \; d\t, $$
for every $t \in [0,T]$, where
$$ \kappa_k := \int_0^T \int_0^1 \left\Vert \partial_1 f(q(\t)+ \theta (q_k(\t)-q(\t)),u_k(\t),\t) - \partial_1 f(q(\t),u(\t),\t) \right\Vert_{\R^n \times \R^n} \; d\theta  d\t .$$
From the continuity and the boundedness of $\partial_1 f$ on $\K_{u,R}$, since $(u_k)_k$ tends to $u$ a.e. on $[0,T]$ and from the classical Lebesgue dominated convergence theorem, one can easily prove that $(\kappa_k)_k$ tends to zero. Finally, from the classical Gronwall lemma, we obtain that $\Vert z_k (t) \Vert_{\R^n} \leq (\rho_k + 2 RT C_{u,R} \kappa_k ) e^{T L_{u,R}}$ for every $t \in [0,T]$. This raises a contradiction with the inequality $\Vert z_k \Vert_\infty \geq \varepsilon$ for all $k \in \N$. The proof is complete.
\end{proof}

\subsection{Application of the Ekeland variational principle}
Let us introduce $ g =(g_i)_{i=1,\ldots,j} : \C_n \to \C_j$ the application defined by $g(q) := G(q,\cdot)$ for every $q \in \C_n$, and let $\S$ be the nonempty closed convex cone of $\C_j$ defined by $\S := \C ([0,T],(\R^-)^j)$. Thus the running state constraints in Problem~\eqref{theproblem} can equivalently be replaced by
$$ g(q) \in \S. $$
Note that $g$ is of class $\C^1$ with $Dg(q)(w) = \partial_1 G (q,\cdot) \times w$ for every $q$, $w \in \C_n$, and that $\S$ has a nonempty interior. \\

Since $(\C_j,\Vert \cdot \Vert_\infty)$ is a separable Banach space, we endow $\C_j$ with an equivalent norm $\Vert \cdot \Vert_{\C_j}$ such that the associated dual norm $\Vert \cdot \Vert_{\C^*_j}$ is strictly convex (see Proposition~\ref{proprenorming} in Appendix~\ref{app1}). Then, we denote by $d_\S$ the $1$-Lipschitz continuous distance function to $\S$ defined by $d_\S (q) := \inf_{z \in \S} \Vert q - z \Vert_{\C_j}$ for every $q \in \C_j$. Since the dual norm $\Vert \cdot \Vert_{\C^*_j}$ is strictly convex, we know that $d_\S$ is strictly Hadamard-differentiable on $\C_j \bs \S$ with $\Vert Dd_\S (q) \Vert_{\C^*_j} = 1$ for every $q \in \C_j \bs \S$ (see Proposition~\ref{propdistance} in Appendix~\ref{app2}). As a consequence, $d^2_\S$ is also strictly Hadamard-differentiable on $\C_j \bs \S$ with $Dd^2_\S (q) = 2 d_\S (q) Dd_\S (q)$ for every $q \in \C_j \bs \S$. We also recall that $d^2_\S$ is Fr\'echet-differentiable on $\S$ with $Dd^2_\S (q) = 0$ for every $q \in \S$ (see Remark~\ref{remarkdistance} in Appendix~\ref{app2}).  \\

In the whole section, let $q(\cdot,u^*) \in \AC$ and $u^* \in \L^\infty$ be an optimal solution of Problem~\eqref{theproblem}. Let $(R_\ell)_\ell$ be a positive sequence such that $R_\ell > \Vert u^* \Vert_{\L^\infty}$ for every $\ell \in \N$ and such that $\lim_\ell R_\ell = +\infty$. Let $(\varepsilon_k)_k$ be a positive sequence such that $\lim_k \varepsilon_k = 0$. For every $\ell$, $k \in \N$, we consider the penalized functional given by
$$ \fonction{J^\ell_k}{\E^\Omega_{u^*,R_\ell}}{\R^+_*}{u}{\sqrt{\Big( \big( \Psi ( q(T,u) ) - \Psi ( q(T,u^*) ) +\varepsilon_k \big)^+ \Big)^2 + d^2_\S \Big( g \big ( q(\cdot,u) \big) \Big)  },} $$
where
$$ \E^\Omega_{u^*,R_\ell} := \{ u \in \E_{u^*,R_\ell} \; \mid \; u(\t) \in \Omega \text{ for a.e. } \t \in [0,T]  \}. $$
Note that $J^\ell_k$ is a positive functional because of the optimality of $u^*$. We endow $\E^\Omega_{u^*,R_\ell}$ with the classical norm $\Vert \cdot \Vert_{\L^1}$. Since $\Omega$ is a nonempty closed subset of $\R^m$, it follows from the (partial) converse of the classical Lebesgue dominated convergence theorem that $(\E^\Omega_{u^*,R_\ell},\Vert \cdot \Vert_{\L^1})$ is a nonempty closed subset of $(\L^1,\Vert \cdot \Vert_{\L^1})$ and consequently $(\E^\Omega_{u^*,R_\ell},\Vert \cdot \Vert_{\L^1})$ is a complete metric space. Moreover, from the continuities of $\Psi$, $F_{u^*,R_\ell}$ (see Proposition~\ref{prop30-1-1}), $d_\S$ and of $g$, one can easily see that $J^\ell_k$ is continuous on $(\E^\Omega_{u^*,R_\ell},\Vert \cdot \Vert_{\L^1})$.  \\

Moreover it holds that $J^\ell_{k} (u^*) = \varepsilon_k$. As a consequence, from the classical Ekeland variational principle, we conclude that for every $\ell$, $k \in \N$, there exists $u^\ell_k \in \E^\Omega_{u^*,R_\ell}$ such that $\Vert u^\ell_k - u^* \Vert_{\L^1} \leq \sqrt{\varepsilon_k}$ and
\begin{equation}\label{eqtresimp}
-\sqrt{\varepsilon_k} \Vert u - u^\ell_k \Vert_{\L^1} \leq J^\ell_k (u) - J^\ell_k (u^\ell_k),
\end{equation}
for all $u \in \E^\Omega_{u^*,R_\ell}$. In particular, for a fixed $\ell \in \N$, note that the sequence $(u^\ell_k)_k$ converges to $u^*$ in $\L^1$ and consequently, the sequence $(q(\cdot,u^\ell_k))_k$ uniformly converges on $[0,T]$ to $q(\cdot,u^*)$ (see Proposition~\ref{prop30-1-1}). \\

For every $\ell$, $k \in \N$, we introduce
$$ \psi^\ell_k := \dfrac{1}{J^\ell_k (u^\ell_k)} \big( \Psi ( q(T,u^\ell_k) ) - \Psi ( q(T,u^*) ) +\varepsilon_k \big)^+ \geq 0, $$
and
$$ \varphi^\ell_k := \left\lbrace \begin{array}{lcr}
\dfrac{1}{J^\ell_k (u^\ell_k)} d_\S \Big( g ( q(\cdot,u^\ell_k) ) \Big) Dd_\S \Big( g \big( q(\cdot,u^\ell_k) \big) \Big) \in \C^*_j & \text{if} & g ( q(\cdot,u^\ell_k) ) \notin \S, \\[20pt]
0 \in \C^*_j & \text{if} & g ( q(\cdot,u^\ell_k) ) \in \S. \\
\end{array} \right. $$
In particular it holds that $\vert \psi^\ell_k \vert^2 + \Vert \varphi^\ell_k \Vert^2_{\C^*_j} = 1$ for every $\ell$, $k \in \N$. \\

In the sequel our aim is to derive some important inequalities from Inequality~\eqref{eqtresimp} using implicit spike variations on $u_k^\ell$.

\begin{remark}\label{remnontrivial1}
In this remark (and in Remarks~\ref{remnontrivial2} and \ref{remnontrivial3}), our aim is to provide two crucial inequalities satisfied by $\varphi^\ell_k $. In the case $g ( q(\cdot,u^\ell_k) ) \notin \S $, recall that $Dd_\S ( g ( q(\cdot,u^\ell_k) ))$ belongs to the subdifferential of $d_\S$ at the point $g ( q(\cdot,u^\ell_k) )$. As a consequence, in both cases $g ( q(\cdot,u^\ell_k) ) \notin \S $ and $g ( q(\cdot,u^\ell_k) ) \in \S $, it holds that
\begin{equation}
\langle \varphi^\ell_k , z - g ( q(\cdot,u^\ell_k) ) \rangle_{\C^*_j \times \C_j} \leq 0,
\end{equation}
for every $z \in \S$. Since $\S$ has a nonempty interior, there exists $\xi \in \S$ and $\delta > 0$ such that $\xi + \delta z \in \S$ for every $z \in \overline{\B}_{(\C_j , \Vert \cdot \Vert_{\C_j})}  (0,1)$. As a consequence, we obtain that
$$ \delta \langle \varphi^\ell_k , z \rangle_{\C^*_j \times \C_j} \leq \langle \varphi^\ell_k , g ( q(\cdot,u^\ell_k) )-\xi \rangle_{\C^*_j \times \C_j} , $$
for every $z \in \overline{\B}_{(\C_j , \Vert \cdot \Vert_{\C_j})} (0,1)$. We deduce that
\begin{equation}
\delta \Vert \varphi^\ell_k \Vert_{\C^*_j} = \delta \sqrt{1- \vert \psi^\ell_k \vert^2} \leq \langle \varphi^\ell_k , g ( q(\cdot,u^\ell_k) ) - \xi \rangle_{\C^*_j \times \C_j} .
\end{equation}
\end{remark}

\subsubsection{First inequality depending on $\ell$ fixed}
In this section, we fix $\ell \in \N$. Recall that the sequence $(u^\ell_k)_k$ converges to $u^*$ in $\L^1$. Using compactness arguments, we infer the existence of a subsequence of $(\varepsilon_k)_k$ (that we do not relabel)\footnote{The subsequence of $(\varepsilon_k)_k$ is not relabel. However, it is worth to note that the extracted subsequence depends on $\ell$ fixed.} such that $(u^\ell_k)_k$ converges to $u^*$ a.e. on $ [0,T]$, $(\psi^\ell_{k})_k$ converges to some $\psi^\ell \geq 0$ and $(\varphi^\ell_{k})_k$ weakly* converges to some $\varphi^\ell \in \C^*_j$. In particular, it holds that $\vert \psi^\ell \vert^2 + \Vert \varphi^\ell \Vert^2_{\C^*_j} \leq 1$. \\

In the whole section, for the ease of notations, we denote by $q^\ell_k := q(\cdot,u^\ell_k)$ for every $k \in \N$. Let $u' \in \L^\infty$ such that $u'(\t) \in \Omega \cap \overline{\B}_{\R^m} (0,R_\ell)$ for a.e. $\t \in [0,T]$. For every $\rho \in [0,1)$, we consider the implicit spike variation
$$ u^\ell_k(\t,\rho) := \left\lbrace \begin{array}{lcr}
u'(\t) & \text{if} & \t \in \Q_\rho, \\
u^\ell_k(\t) & \text{if} & \t \notin \Q_\rho,
\end{array} \right. $$
for a.e. $\t \in [0,T]$, where $\Q_\rho$ is defined in Lemma~\ref{lem123684} associated to $h_{u^\ell_k,u'} \in \L^\infty_n \subset \L^1_n$ defined by
$$ h_{u^\ell_k,u'} (\t) := f(q^\ell_k(\t),u'(\t),\t)-f(q^\ell_k(\t),u^\ell_k(\t),\t) , $$
for a.e. $\t \in [0,T]$.\footnote{For $\rho = 0$, we fix $\Q_\rho = \emptyset$.} \\

First of all, note that $\Vert u^\ell_k (\cdot,\rho) - u^* \Vert_{\L^1} \leq \Vert u^\ell_k (\cdot,\rho) - u^\ell_k \Vert_{\L^1} + \Vert u^\ell_k - u^* \Vert_{\L^1} \leq 2 R_\ell T \rho + \sqrt{\varepsilon_k} < \nu_{u^*,R_\ell}$ for sufficiently small $\rho$ and sufficiently large $k$ and then $u^\ell_k (\cdot,\rho) \in \E^\Omega_{u^*,R_\ell}$. For such a sufficiently small $\rho$ and sufficiently large $k$, we apply Inequality~\eqref{eqtresimp} with $u = u^\ell_k (\cdot,\rho)$ and we obtain that
$$ - 2R_\ell T \sqrt{\varepsilon_k} \leq \dfrac{J^\ell_k (u^\ell_k (\cdot,\rho)) - J^\ell_k (u^\ell_k)}{\rho} = \dfrac{J^\ell_k (u^\ell_k (\cdot,\rho))^2 - J^\ell_k (u^\ell_k)^2}{\rho} \times \dfrac{1}{J^\ell_k (u^\ell_k (\cdot,\rho)) + J^\ell_k (u^\ell_k)}. $$
From the continuity of $J^\ell_k$, we get that $\lim_{\rho \to 0} J^\ell_k (u^\ell_k (\cdot,\rho)) + J^\ell_k (u^\ell_k) = 2 J^\ell_k (u^\ell_k)$. From the differentiabilities of the application $x \mapsto (x^+)^2$ for $x \in \R$, of $\Psi$, of $g$, of $d^2_\S$ and of $F_{u^\ell_k,u'}$ (see Proposition~\ref{prop984357}), we obtain that
\begin{multline*}
\lim\limits_{\rho \to 0} \dfrac{J^\ell_k (u^\ell_k (\cdot,\rho))^2 - J^\ell_k (u^\ell_k)^2}{\rho} \\
= 2 \Big( \Psi ( q^\ell_k(T) ) - \Psi ( q(T,u^*) ) + \varepsilon_k \Big)^+ \Big\langle \nabla \Psi ( q^\ell_k(T) ) ,  w(T,u^\ell_k,u') \Big\rangle_{\R^n \times \R^n} \\
+ \Big\langle 2 d_\S ( g ( q^\ell_k )) Dd_\S ( g ( q^\ell_k ) ) , Dg ( q^\ell_k ) ( w(\cdot ,u^\ell_k,u') ) \Big\rangle_{\C^*_j \times \C_j},
\end{multline*}
with the convention that the second term is zero if $g( q^\ell_k ) \in \S$. Finally, we have obtained that
\begin{equation}\label{eq512b}
- 2R_\ell T \sqrt{\varepsilon_k} \leq \psi^\ell_k \Big\langle \nabla \Psi ( q^\ell_k(T) ) ,  w(T,u^\ell_k,u') \Big\rangle_{\R^n \times \R^n} + \Big\langle \varphi^\ell_k , Dg ( q^\ell_k ) ( w(\cdot ,u^\ell_k,u') ) \Big\rangle_{\C^*_j \times \C_j} .
\end{equation}
To conclude this section, we need the following result.\footnote{This result requires to fix $\ell \in \N$. Indeed, one needs a bound on $\Vert u^\ell_k \Vert_{\L^\infty}$ in order to conclude from the classical Lebesgue dominated convergence theorem.}

\begin{lemma}
The sequence $(w(\cdot,u^\ell_k,u'))_k$ uniformly converges on $[0,T]$ to $w(\cdot,u^*,u')$.
\end{lemma}

\begin{proof}
In this proof, for the ease of notations, we denote by $q^* :=q(\cdot,u^*)$, $w := w(\cdot,u^*,u')$ and by $w_k := w(\cdot, u^\ell_k ,u')$ for all $k \in \N$. It holds that
\begin{multline*}
w_k (t) - w(t) = \int_0^t \partial_1 f (q^\ell_k (\t),u^\ell_k(\t),\t) \times w_k (\t) + h_{u^\ell_k,u'}(\t) \\
- \partial_1 f (q^* (\t),u^*(\t),\t) \times w (\t) - h_{u^*,u'}(\t) \; d\t ,
\end{multline*}
that is,
\begin{multline*}
w_k (t) - w(t) = \int_0^t \partial_1 f (q^\ell_k (\t),u^\ell_k(\t),\t) \times (w_k (\t) -w(\t)) \; d\t + \int_0^t  h_{u^\ell_k,u'}(\t) -  h_{u^*,u'}(\t) \; d\t \\[3pt]
 + \int_0^t \Big( \partial_1 f (q^\ell_k (\t),u^\ell_k(\t),\t) - \partial_1 f (q^* (\t),u^*(\t),\t) \Big) \times w(\t) \; d\t ,
\end{multline*}
for every $t \in [0,T]$. Recall that $u^\ell_k \in \E_{u^*,R_\ell}$ and $u' \in \overline{\B}_{\L^\infty} (0,R_\ell)$, then $(q^\ell_k(\t),u^\ell_k(\t),\t) \in \K_{u^*,R_\ell}$ (see Proposition~\ref{propEUR}) and $(q^\ell_k(\t),u'(\t),\t) \in \K_{u^*,R_\ell}$ for a.e. $\t \in [0,T]$ and recall that $f$ and $\partial_1 f$ are bounded on $\K_{u^*,R_\ell}$ by $L_{u^*,R_\ell} \geq 0$. Recall also that $(u^\ell_k)_k$ tends to $u^*$ a.e. on $[0,T]$. Finally, using similar arguments than in the proof of Proposition~\ref{prop984357} and the classical Gronwall lemma, one can easily conclude the proof.
\end{proof}

Using the above lemma and the $\C^1$-regularity of $\Psi$ and $g$, by letting $k$ tend to $+\infty$ in Inequality~\eqref{eq512b}, we obtain that
\begin{equation}\label{firstineq}
0 \leq \psi^\ell \Big\langle \nabla \Psi ( q(T,u^*) ) ,  w(T,u^*,u') \Big\rangle_{\R^n \times \R^n} + \Big\langle \varphi^\ell , Dg ( q(\cdot,u^*) ) ( w(\cdot ,u^*,u') ) \Big\rangle_{\C^*_j \times \C_j}  .
\end{equation}

\begin{remark}\label{remnontrivial2}
Letting $k$ tend to $+\infty$ in Remark~\ref{remnontrivial1}, one can easily obtain the two following crucial inequalities:
\begin{equation}
\langle \varphi^\ell , z - g ( q(\cdot,u^*) ) \rangle_{\C^*_j \times \C_j} \leq 0,
\end{equation}
for every $z \in \S$, and
\begin{equation}
\delta \sqrt{1- \vert \psi^\ell \vert^2} \leq \langle \varphi^\ell , g ( q(\cdot,u^*) ) - \xi \rangle_{\C^*_j \times \C_j} .
\end{equation}
\end{remark}

\subsubsection{Second inequality independent of $\ell$}
In the previous section, we have obtained Inequality~\eqref{firstineq} that is valid for a fixed $\ell \in \N$ and for every $u' \in \L^\infty$ such that $u'(\t) \in \Omega \cap \overline{\B}_{\R^m}(0,R_\ell)$. Our aim in this section is to remove the dependence in~$R_\ell$ (in order to cover the case where $\Omega$ is unbounded). \\

Since $\vert \psi^{\ell} \vert^2 + \Vert \varphi^{\ell} \Vert^2_{\C^*_j} \leq 1$ for every $\ell \in \N$ and from compactness arguments, we infer the existence of a subsequence of $(R_\ell)_\ell$ (that we do not relabel) such that $(\psi^{\ell})_\ell$ converges to some $\psi \geq 0$ and $(\varphi^{\ell})_\ell$ weakly* converges to some $\varphi \in \C^*_j$. \\

Let $u' \in \L^\infty$ such that $u'(\t) \in \Omega$ for a.e. $\t \in [0,T]$. Let $\ell \in \N$ be sufficiently large in order to have $R_\ell > \Vert u' \Vert_{\L^\infty}$.  From Inequality~\eqref{firstineq}, it holds that
$$ 0 \leq \psi^{\ell}\Big\langle \nabla \Psi ( q(T,u^*) ) ,  w(T,u^*,u') \Big\rangle_{\R^n \times \R^n} + \Big\langle \varphi^{\ell} , Dg ( q(\cdot,u^*) ) ( w(\cdot ,u^*,u') ) \Big\rangle_{\C^*_j \times \C_j}  .$$
Letting $\ell $ tend to $+\infty$, we prove that
\begin{equation}\label{secondineq}
0 \leq \psi \Big\langle \nabla \Psi ( q(T,u^*) ) ,  w(T,u^*,u') \Big\rangle_{\R^n \times \R^n} + \Big\langle \varphi , Dg ( q(\cdot,u^*) ) ( w(\cdot ,u^*,u') ) \Big\rangle_{\C^*_j \times \C_j} ,
\end{equation}
for every $u' \in \L^\infty$ such that $u'(\t) \in \Omega$ for a.e. $\t \in [0,T]$.

\begin{remark}\label{remnontrivial3}
Letting $\ell $ tend to $+\infty$ in Remark~\ref{remnontrivial2}, one can easily obtain the two following crucial inequalities:
\begin{equation}\label{eqsat1psi}
\langle \varphi , z - g ( q(\cdot,u^*) ) \rangle_{\C^*_j \times \C_j} \leq 0,
\end{equation}
for every $z \in \S$, and
\begin{equation}\label{eqsat2psi}
\delta \sqrt{1- \vert \psi \vert^2} \leq \langle \varphi , g ( q(\cdot,u^*) ) - \xi \rangle_{\C^*_j \times \C_j} .
\end{equation}
Inequality~\eqref{eqsat2psi} proves that the couple $(\psi,\varphi)$ is not trivial. 
\end{remark}

\subsection{Introduction of the adjoint vector $p$}

\subsubsection{Introduction of $\eta$}
Let us denote by $\varphi = (\varphi_i)_{i=1,\ldots,j}$ where $\varphi_i \in \C^*_1$ and let us apply the classical Riesz theorem (see Proposition~\ref{propriesz} in Appendix~\ref{annexeBV1}). For every $i=1,\ldots,j$, there exists a unique $\eta_i \in \NBV_1$ such that
$$ \langle \varphi_i , z \rangle_{\C^*_1 \times \C_1} = \int_0^T  z(\t) \; d\eta_i (\t) , $$
for every $z \in \C_1$. Recall that $\varphi_i = 0$ if and only if $\eta_i = 0$. As a consequence, from Remark~\ref{remnontrivial3}, the couple $(\psi,\eta)$ is not trivial, where $\eta := (\eta_i)_{i=1,\ldots,j} \in \NBV_j$. \\

Taking 
$$ z = \Big( g_1 ( q(\cdot,u^*) ) , \ldots , g_{i-1} ( q(\cdot,u^*) ) , 0 , g_{i+1} ( q(\cdot,u^*) ), \ldots , g_{j} ( q(\cdot,u^*) ) \Big) \in \S $$
and
$$ z = \Big( g_1 ( q(\cdot,u^*) ) , \ldots , g_{i-1} ( q(\cdot,u^*) ) , 2 g_{i} ( q(\cdot,u^*) ) , g_{i+1} ( q(\cdot,u^*) ), \ldots , g_{j} ( q(\cdot,u^*) ) \Big) \in \S $$
in Inequality~\eqref{eqsat1psi}, we obtain that $\langle \varphi_i , g_i ( q(\cdot,u^*) ) \rangle_{\C^*_1 \times \C_1} = 0$, that is,
$$ \int_0^T G_i (q(\t,u^*),\t ) \; d\eta_i (\t) = 0, $$
for every $i=1,\ldots,n$. \\

Moreover, it follows that $\langle \varphi_i , z \rangle_{\C^*_1 \times \C_1} \geq 0$ for every $z \in \C_1^+$. From the classical Riesz theorem (see Proposition~\ref{propriesz} in Appendix~\ref{annexeBV1}), we deduce that $\eta_i$ is monotically increasing on $[0,T]$ for every $i=1,\ldots,j$.

\subsubsection{Definition of $p$}
Using notations introduced in Appendix~\ref{annexeBVnotations}, one has
$$ \langle \varphi , z \rangle_{\C^*_j \times \C_j} = \int_0^T \langle z(\t) , d\eta (\t) \rangle $$ 
for every $z \in \C_j$. From Inequality~\eqref{secondineq} and since $Dg(q)(w) = \partial_1 G(q,\cdot) \times w$ for every $q$, $w \in \C_n$, we have proved that
\begin{equation}\label{thirdineq}
0 \leq \psi \Big\langle \nabla \Psi ( q(T,u^*) ) ,  w(T,u^*,u') \Big\rangle_{\R^n \times \R^n} + \int_0^T \Big\langle \partial_1 G ( q(\t,u^*),\t ) \times w (\t,u^*,u') ,  d\eta (\t) \Big\rangle,
\end{equation}
for every $u' \in \L^\infty$ such that $u'(\t) \in \Omega$ for a.e. $\t \in [0,T]$. \\

Let $Z(\cdot,\cdot)$ be the state-transition matrix associated to $\partial_1 f(q(\cdot,u^*),u^*,\cdot) \in \L^\infty([0,T],\R^{n,n})$ (see Appendix~\ref{appstatetransition}). From the classical Duhamel formula (see Proposition~\ref{propduhamelclassique} in Appendix~\ref{appcauchyproblemclassique}), it holds that
\begin{equation*}
w(t,u^*,u') = \int_0^t Z(t,s) \times h_{u^*,u'} (s) \; ds
\end{equation*}
for every $t \in [0,T]$. Replacing $w(\cdot,u^*,u')$ in \eqref{thirdineq}, using first the Fubini-type formula~\eqref{eq326} and then Equality~\eqref{eq7645}, one can obtain that
\begin{multline}\label{eq698235}
0 \leq \int_0^T \Big\langle h_{u^*,u'}(s) , \psi Z(T,s)^\top \times \nabla \Psi ( q(T,u^*) ) \\
 + \int_s^T Z(\t,s)^\top \times \partial_1 G( q(\t,u^*),\t)^\top \times d\eta (\t) \Big\rangle_{\R^n \times \R^n} \; ds,
\end{multline}
for every $u' \in \L^\infty$ such that $u'(\t) \in \Omega$ for a.e. $\t \in [0,T]$. \\

Let $p \in \BV_n$ be the unique global solution of the backward linear Cauchy-Stieltjes problem given by
$$ \left\lbrace \begin{array}{l}
- dp = \partial_1 f ( q(\cdot,u^*),u^*,\cdot )^\top \times p \; dt + \sum_{i=1}^j \partial_1 G_i (q(\cdot,u^*),\cdot) \; d\eta_i , \quad \text{on $[0,T]$,}  \\[5pt]
p(T)= \psi \nabla \Psi ( q(T,u^*) ).
\end{array} \right. $$
We refer to Proposition~\ref{propexistCSP} in Appendix~\ref{appderniereCSP} for the existence and uniqueness of $p$. Note that $p$ is independent of $u'$. From the Duhamel-type formula (see Proposition~\ref{propduhamelstieltjes} in Appendix~\ref{appderniereCSP}), it holds that
\begin{equation*}
p(s) = \psi Z(T,s)^\top \times \nabla \Psi ( q(T,u^*) ) + \int_s^T Z(\t,s)^\top \times \partial_1 G( q(\t,u^*),\t)^\top \times d\eta (\t) \in \R^n ,
\end{equation*}
for every $s \in [0,T]$. It follows from the above expression of $p$ and from Inequality~\eqref{eq698235} that $ \int_0^T \langle h_{u^*,u'}(s) , p(s) \rangle_{\R^n \times \R^n} ds \geq 0  $, that is,
\begin{equation}\label{eq821}
\int_0^T H(q(s,u^*),u'(s),p(s),s) - H(q(s,u^*),u^*(s),p(s),s) \; ds \geq 0,
\end{equation}
for every $u' \in \L^\infty$ such that $u'(\t) \in \Omega$ for a.e. $\t \in [0,T]$.

\subsection{End of the proof}
Let $v \in \Omega$ be fixed. Let $t \in [0,T)$ be a continuity point of $p \in \BV_n$ and be a Lebesgue point of the application $s \mapsto H(q(s,u^*),u^*(s),p(s),s)$ which belongs to $\L^\infty_1$. Let $\alpha \in (0,T-t)$ and let us consider
$$ u'(\t) := \left\lbrace \begin{array}{lcr}
v & \text{if} & \t \in [t,t+\alpha), \\
u^*(\t) & \text{if} & \t \notin [t,t+\alpha),
\end{array} \right. $$
for a.e. $\t \in [0,T]$. From Inequality~\eqref{eq821}, it holds that
$$ \int_t^{t+\alpha} H(q(s,u^*),v,p(s),s) - H(q(s,u^*),u^*(s),p(s),s) \; ds \geq 0. $$
Dividing by $\alpha > 0$ and letting $\alpha \to 0^+$, we obtain that
$$ H(q(t,u^*),v,p(t),t) - H(q(t,u^*),u^*(t),p(t),t) \geq 0. $$
Since the last inequality is true for every $v \in \Omega$ and for a.e. $t \in [0,T]$, we obtain the maximization condition
$$ u^*(t) \in \argmin_{v \in \Omega}  H(q(t,u^*),v,p(t),t) $$
for a.e. $t \in [0,T]$.

\appendix

\section{Proof of Lemma~\ref{lem123684}}\label{appunity}
Recall that the classical Lesbesgue measure $\lambda$ is a nonatomic measure (see, \textit{e.g.}, \cite[Remark 1.161 p.111]{fons}). As a consequence, from the classical Sierpinski (or Lyapunov) theorem (see \cite{sier} or \cite[p.37]{frys}), for all measurable set $\RR \subset [0,T]$, there exists a measurable set $\RR_\rho \subset \RR$ such that $\lambda (\RR_\rho) = \rho \lambda (\RR)$ for all $\rho \in (0,1)$. \\

The whole section is dedicated to the proof of Lemma~\ref{lem123684}. Let $\rho \in (0,1)$.

\begin{lemma}\label{lem876423}
Let $b : [0,T]^2 \to \R^n$ be defined by
$$ b(t,s) := h(s) \mathbf{1}_{[0,t]}(s). $$
Then, $b \in \C([0,T],\L^1_n)$.
\end{lemma}

\begin{proof}
Let $t \in [0,T]$ and let $(t_n) \subset [0,T]$ be a decreasing sequence such that $t_n \to t$. Then, it holds that
\begin{multline*}
\Vert b(t_n,\cdot) - b(t,\cdot) \Vert_{\L^1} = \int_0^T \vert b(t_n,s)-b(t,s) \vert \; ds \\ = \int_0^t  \vert b(t_n,s)-b(t,s) \vert \; ds + \int_t^{t_n}  \vert b(t_n,s)-b(t,s) \vert \; ds = \int_t^{t_n} \vert h(s) \vert \; ds \to 0.
\end{multline*}
Similarly, we prove that $\Vert b(t_n,\cdot) - b(t,\cdot) \Vert_{\L^1} \to 0$ for any increasing sequence $(t_n) \subset [0,T]$ such that $t_n \to t$. The proof is complete.
\end{proof}

Since $[0,T]$ is compact, there exists $\delta > 0$ such that $\Vert b(t,\cdot) - b(\bar{t},\cdot) \Vert_{\L^1} \leq \frac{ \rho^2}{2(\rho+1)}$ for all $t$, $\bar{t} \in [0,T]$ satisfying $\vert t - \bar{t} \vert < \delta$. In the sequel, we fix $0 = t_0 < t_1 < \ldots < t_N = T$ such that $\vert t_{r+1}-t_r \vert < \delta$ for all $r=0,\ldots,N-1$ and we define
$$ B(\cdot) := \Big( b(t_0,\cdot), b(t_1,\cdot), \ldots , b(t_N,\cdot) \Big) \in \L^1 ([0,T],(\R^n)^{N+1}). $$

\begin{lemma}\label{lem641987}
There exists a measurable set $\Q_\rho \subset [0,T]$ such that $\lambda (\Q_\rho) = \rho T$ and
$$ \left\Vert \int_0^T \left( 1 - \frac{1}{\rho} \mathbf{1}_{\Q_\rho} (s) \right) B(s) \; ds \right\Vert_{(\R^n)^{N+1}} \leq \frac{\rho}{2}. $$
\end{lemma}

\begin{proof}
Since $B \in \L^1 ([0,T],(\R^n)^{N+1})$, there exists a simple function $J : [0,T] \to (\R^n)^{N+1}$ such that $\int_0^T \Vert B(s) - J(s) \Vert_{(\R^n)^{N+1}} ds \leq \frac{\rho^2}{2(\rho +1)}$. Let us denote by $J := \sum_{i=1}^K a_i \mathbf{1}_{\RR^i}$, where $a_i \in (\R^n)^{N+1}$ and $\RR^i \subset [0,T]$ are measurable sets such that $\coprod_{i=1}^K \RR^i = [0,T]$. Since $\lambda$ is nonatomic, there exist $\RR^i_\rho \subset \RR^i$ such that $\lambda (\RR^i_\rho) = \rho \lambda (\RR^i)$ for all $i=1,\ldots,K$. Let us define $\Q_\rho := \coprod_{i=1}^K \RR^i_\rho \subset [0,T]$. Note that $\lambda (\Q_\rho ) = \rho T$. Moreover, it holds that
\begin{multline*}
 \left\Vert \int_0^T \left( 1 - \frac{1}{\rho} \mathbf{1}_{\Q_\rho} (s) \right) B(s) \; ds \right\Vert_{(\R^n)^{N+1}} \\
 \leq   \left\Vert \int_0^T \left( 1 - \frac{1}{\rho} \mathbf{1}_{\Q_\rho} (s) \right) J(s) \; ds \right\Vert_{(\R^n)^{N+1}} +  \left\Vert \int_0^T \left( 1 - \frac{1}{\rho} \mathbf{1}_{\Q_\rho} (s) \right) (B(s)-J(s)) \; ds \right\Vert_{(\R^n)^{N+1}} .
\end{multline*}
The second integral in the left term can be easily bounded by $\frac{\rho}{2}$ and the first one is equal to
$$  \left\Vert \sum_{i=1}^K a_i \int_0^T \left( \mathbf{1}_{\RR^i} (s) - \frac{1}{\rho} \mathbf{1}_{\Q_\rho \cap \RR^i} (s) \right) \; ds \right\Vert_{(\R^n)^{N+1}} = \left\Vert \sum_{i=1}^K a_i \left( \lambda (\RR^i) - \frac{1}{\rho} \lambda (\RR^i_\rho) \right)  \right\Vert_{(\R^n)^{N+1}} = 0.$$
The proof is complete.
\end{proof}

Let us now conclude the proof of Lemma~\ref{lem123684}. Let $t \in [0,T]$. There exists $r \in \{ 0,\ldots,N-1 \}$ such that $t \in [t_r,t_{r+1}]$. In particular, it holds that $\vert t-t_r \vert < \delta$ and thus $\Vert b(t,\cdot) - b(t_r,\cdot) \Vert_{\L^1} \leq \frac{\rho^2}{2(\rho +1)}$ (see remark after Lemma~\ref{lem876423}). It holds that
\begin{multline*}
\left\Vert \int_0^t \left( 1 - \frac{1}{\rho} \mathbf{1}_{\Q_\rho} (s) \right) h(s) \; ds \right\Vert_{\R^n} = \left\Vert \int_0^T \left( 1 - \frac{1}{\rho} \mathbf{1}_{\Q_\rho} (s) \right) b(t,s) \; ds \right\Vert_{\R^n} \\ 
\leq  \left\Vert \int_0^T \left( 1 - \frac{1}{\rho} \mathbf{1}_{\Q_\rho} (s) \right) (b(t,s)-b(t_r,s)) \; ds \right\Vert_{\R^n} + \left\Vert \int_0^T \left( 1 - \frac{1}{\rho} \mathbf{1}_{\Q_\rho} (s) \right) b(t_r,s) \; ds \right\Vert_{\R^n} .
\end{multline*}
The first term can be bounded by $(1 + \frac{1}{\rho}) \Vert b(t,\cdot) - b(t_r,\cdot) \Vert_{\L^1} \leq \frac{\rho}{2 }$ and the second one can be bounded by $ \left\Vert \int_0^T \left( 1 - \frac{1}{\rho} \mathbf{1}_{\Q_\rho} (s) \right) B(s)  ds \right\Vert_{(\R^n)^{N+1}} \leq \frac{\rho}{2} $ (see Lemma~\ref{lem641987}). Finally, we have proved that
$$ \left\Vert \int_0^t \left( 1 - \frac{1}{\rho} \mathbf{1}_{\Q_\rho} (s) \right) h(s) \; ds \right\Vert_{\R^n} \leq \rho .$$
The proof of Lemma~\ref{lem123684} is complete.

\section{Some recalls about Banach spaces geometry}

\subsection{Renorming a separable Banach space}\label{app1}

Let $(X,\Vert \cdot \Vert)$ be a normed linear space. The \textit{dual space} of $(X,\Vert \cdot \Vert)$ is $X^* := \mathcal{L}( (X,\Vert \cdot \Vert), \R)$ and $X^*$ can be endowed with the \textit{dual norm of $\Vert \cdot \Vert$} defined by
$$ \fonction{\Vert \cdot \Vert_*}{X^*}{\R^+}{f}{\sup\limits_{ \substack{ x \in X \\ \Vert x \Vert \leq 1 } } \; \vert \langle f , x \rangle_{X^* \times X} \vert  .} $$
In this case, we denote $(X^*,\Vert \cdot \Vert_*) = \mathrm{dual} ( (X,\Vert \cdot \Vert ))$. Recall that $(X^* , \Vert \cdot \Vert_*)$ is a Banach space, even if $(X,\Vert \cdot \Vert)$ is not.

\begin{lemma}\label{lem69876}
Let $(X,\Vert \cdot \Vert)$ be a Banach space\footnote{The Banach assumption is necessary in order to apply \cite[Proposition 3.13 p.63]{brez} deriving from the classical Banach-Steinhaus theorem.} and $(X^* ,\Vert \cdot \Vert_*) = \mathrm{dual} ( (X,\Vert \cdot \Vert ))$. Let $N^*$ be a norm on $X^*$ equivalent to $\Vert \cdot \Vert_*$. Then, the following properties are equivalent:
\begin{enumerate}
\item $N^*$ is weak* lower semicontinuous on $(X^* ,\Vert \cdot \Vert_*)$;
\item There exists $N$ a norm on $X$ equivalent to $\Vert \cdot \Vert$ such that $(X^*,N^*) = \mathrm{dual}((X,N))$.
\end{enumerate}
\end{lemma}

\begin{remark}
This lemma is given in \cite[Lemma 3.94 p.251]{penot}.
\end{remark}

\begin{proof}
$2 \Rightarrow 1.$ Since $(X,\Vert \cdot \Vert)$ is a Banach space and since $N$ is equivalent to $\Vert \cdot \Vert$, note that $(X,N)$ is a Banach space. Let us prove that $N^*$ is weak* lower semicontinuous on $(X^* ,\Vert \cdot \Vert_*)$. Let $(f_n) \subset X^*$ and $f \in X^*$. From \cite[Proposition 3.13 p.63]{brez}, $(f_n)$ weak* converges to $f$ in $(X^* ,\Vert \cdot \Vert_*) = \mathrm{dual} ((X,\Vert \cdot \Vert))$ if and only if $\langle f_n , x \rangle_{X^* \times X}$ tends to $\langle f , x \rangle_{X^* \times X}$ for every $x \in X$ if and only if $(f_n)$ weak* converges to $f$ in $(X^* ,N^*) = \mathrm{dual} ((X,N))$. Since $(X,N)$ is a Banach space, we obtain from \cite[Proposition 3.13 p.63]{brez} that $(N^*(f_n))$ is bounded and $N^* (f) \leq \liminf N^* (f_n)$. \\

$1 \Rightarrow 2.$ We know that there exist $0 < \mu_1 \leq \mu_2$ such that $\mu_1 \Vert f \Vert_* \leq N^* (f) \leq \mu_2 \Vert f\Vert_*$ for every $f \in X^*$. Our proof is based on two steps. \\

\textit{First step - Definition of $N$ and equivalence to $\Vert \cdot \Vert$.} We define
$$ \fonction{N}{X}{\R^+}{x}{\sup\limits_{ \substack{ f \in X^* \\ N^* (f) \leq 1 } } \; \vert \langle f , x \rangle_{X^* \times X} \vert  .} $$ 
First of all, let us note that $N$ is well-defined since $\vert \langle f , x \rangle_{X^* \times X} \vert \leq \Vert f \Vert_* \Vert x \Vert \leq \frac{1}{\mu_1} N^* (f) \Vert x \Vert \leq \frac{1}{\mu_1} \Vert x \Vert$ for every $x \in X$ and every $f \in X^*$ such that $N^*(f) \leq 1$. In particular it holds that $N(x) \leq \frac{1}{\mu_1} \Vert x \Vert$ for every $x \in X$. Let us prove that $N$ is a norm on $X$. Clearly we have $N(\lambda x) = \vert \lambda \vert N (x)$ and $N(x+y) \leq N(x)+N(y)$ for every $x$, $y \in X$ and every $\lambda \in \R$. Moreover, it holds that $N(0)=0$. Now let us consider $x \neq 0$. From the classical Hahn-Banach theorem (see \cite[Corollary 1.3 p.3]{brez}) applied to $(X,\Vert \cdot \Vert)$ (with dual $(X^*,\Vert \cdot \Vert_*)$), there exists $f \in X^*$ such that $\langle f , x \rangle_{X^* \times X} = \Vert x \Vert^2$ and $\Vert f \Vert_* = \Vert x \Vert$. Let $g := \frac{f}{\mu_2 \Vert x \Vert} \in X^*$. It holds that $N^* (g) = \frac{N^* (f)}{\mu_2 \Vert x \Vert} \leq \frac{\mu_2 \Vert f \Vert_*}{\mu_2 \Vert x \Vert} = 1$ and $\vert \langle g , x \rangle_{X^* \times X} \vert = \frac{\vert \langle f , x \rangle_{X^* \times X} \vert}{\mu_2 \Vert x \Vert} = \frac{\Vert x \Vert}{\mu_2} $. Thus $N(x) \geq \frac{\Vert x \Vert}{\mu_2} > 0$. As a consequence, we have proved that $N(x) = 0$ if and only if $x=0$. We conclude that $N$ is a norm on $X$. Moreover, we have also proved that $\frac{1}{\mu_2} \Vert x \Vert  \leq N(x) \leq \frac{1}{\mu_1} \Vert x \Vert $ for every $x \in X$ (the case $x=0$ is obvious). As a consequence, $N$ is equivalent to $\Vert \cdot \Vert$. \\

\textit{Second step - $(X^*,N^*) = \mathrm{dual}((X,N))$.} Since $N$ is equivalent to $\Vert \cdot \Vert$, $\mathcal{L}( (X,N), \R) = \mathcal{L}( (X,\Vert \cdot \Vert), \R) = X^*$, \textit{i.e.} $X^*$ is the dual space of $(X,N)$. Let us introduce $\tilde{N}$ the dual norm of $N$ on $X^*$ given by
$$ \fonction{\tilde{N}}{X^*}{\R^+}{f}{\sup\limits_{ \substack{ x \in X \\ N (x) \leq 1 } } \; \vert \langle f , x \rangle_{X^* \times X} \vert  .} $$
In particular, we have $(X^*,\tilde{N}) =\mathrm{dual}((X,N))$. Our aim is to prove that $N^* = \tilde{N}$. Firstly, let $f \in X^*$ such that $f \neq 0$. Let $x \in X$ such that $N(x) \leq 1$. From the definition of $N$, since $N^* (\frac{f}{N^* (f)} ) = 1$, we have $\vert \langle \frac{f}{N^* (f)} , x \rangle_{X^* \times X} \vert \leq N(x) \leq 1$. Thus, $\vert \langle f , x \rangle_{X^* \times X} \vert \leq N^*(f) $ for every $x \in X$ such that $N(x) \leq 1$. As a consequence, from the definition of $\tilde{N}$, we obtain that $\tilde{N}(f) \leq N^* (f)$ for every $f \in X^*$ (the case $f=0$ is obvious). Secondly, let us assume by contradiction that there exists $f_0 \in X^*$ such that $\tilde{N}(f_0) < N^* (f_0)$. Thus $f_0 \neq 0$ and we can define $g := \frac{f_0}{\tilde{N}(f_0)} \in X^*$ satisfying $\tilde{N} (g) = 1 < N^* (g)$. Hence $g \notin \overline{\B}_{(X^*,N^*)} (0,1)$. From Hypothesis~$1$, one can easily see that $\overline{\B}_{(X^*,N^*)} (0,1)$ is a nonempty weakly* closed convex of $(X^*,\Vert \cdot \Vert_*)$. Then, recall that the dual of $(X^*,\Vert \cdot \Vert_*) = \mathrm{dual}((X,\Vert \cdot \Vert))$ endowed with the classical weak* topology is the set $(\langle \cdot , x \rangle_{X^* \times X})_{x \in X}$ (see \cite[Proposition 3.14 p.64]{brez}). Finally, from the classical Hahn-Banach theorem\footnote{The classical Hahn-Banach theorem can be applied here since $(X^*,\Vert \cdot \Vert_*) = \mathrm{dual}((X,\Vert \cdot \Vert))$ endowed with the classical weak* topology is a topological vector space that is locally convex.} (see \cite[Theorem 1.7 p.7]{brez}), there exists $x \in X$, $\alpha \in \R$ and $\varepsilon > 0$ such that $\langle f , x \rangle_{X^* \times X} \leq \alpha - \varepsilon < \alpha + \varepsilon \leq \langle g , x \rangle_{X^* \times X} $ for every $f \in \overline{\B}_{(X^*,N^*)} (0,1)$. As a consequence, from the definition of $N$, we obtain that $N(x) < \langle g , x \rangle_{X^* \times X} \leq \tilde{N}(g) N(x) $ and finally $\tilde{N}(g) > 1$. This is a contradiction.
\end{proof}

\begin{proposition}\label{proprenorming}
Let $(X,\Vert \cdot \Vert)$ be a separable Banach space and $(X^* ,\Vert \cdot \Vert_*) = \mathrm{dual} ( (X,\Vert \cdot \Vert ))$. There exists a norm $N$ on $X$ equivalent to $\Vert \cdot \Vert$ such that:
\begin{enumerate}
\item $N^*$ is equivalent on $X^*$ to $\Vert \cdot \Vert_*$;
\item $N^*$ is strictly convex on $X^*$;
\end{enumerate} 
where $(X^*,N^*) = \mathrm{dual}((X,N))$.
\end{proposition}

\begin{remark}
This proposition is given in \cite[Theorem 2.18 p.42]{liyong}.
\end{remark}

\begin{proof}
Our proof is based on four steps, including the application of Lemma~\ref{lem69876}. \\

\textit{First step (definition of $N^*$ and equivalence to $\Vert \cdot \Vert_*$).} Let $(e_k)_{k \in \N^*} \subset X$ be a sequence dense in $\overline{\B}_{(X,\Vert \cdot \Vert)}(0,1)$. Then, let us consider the linear operator given by
$$ \fonction{A}{X^*}{\ell^2(\N^*,\R)}{f}{\left( \dfrac{1}{2^{k/2}} \langle f , e_k \rangle_{X^* \times X} \right)_{k \in \N^*}.} $$
Let us note that
$$ \Vert A(f) \Vert^2_{\ell^2} = \sum_{k \in \N^*} \dfrac{1}{2^k} \vert \langle f , e_k \rangle_{X^* \times X} \vert^2 \leq \Vert f \Vert^2_* ,$$
for every $f \in X^*$. Now we define
$$ \fonction{N^*}{X^*}{\R^+}{f}{\Vert f \Vert_* + \Vert A(f) \Vert_{\ell^2} .} $$
Clearly $N^*$ is a norm on $X^*$ that is equivalent to $\Vert \cdot \Vert_*$ since $\Vert f \Vert_* \leq N^* (f) \leq 2 \Vert f \Vert_*$ for every $f \in X^*$. \\

\textit{Second step (strict convexity of $N^*$).} Let $f$, $g \in X^*$ and let $\lambda \in (0,1)$ such that $N^* (\lambda f + (1- \lambda) g) = \lambda N^* (f) + (1- \lambda ) N^* (g)$. Our aim is to prove that $f = \mu g$ for some $\mu \geq 0$. Since $\Vert \cdot \Vert_*$ and $\Vert A(\cdot) \Vert_{\ell^2}$ are convex, we get that $\Vert A (\lambda f + (1- \lambda) g) \Vert_{\ell^2} = \lambda \Vert A(f) \Vert_{\ell^2} + (1- \lambda ) \Vert A(g) \Vert_{\ell^2}$. Computing the square of the previous equality and using the linearity of $A$ and the bilinearity of $\langle \cdot , \cdot \rangle_{\ell^2}$, one can easily obtain that $ \langle A(f) , A(g) \rangle_{\ell^2} = \Vert A ( f ) \Vert_{\ell^2} \Vert A (g) \Vert_{\ell^2} $. From the classical Cauchy-Schwarz inequality in $\ell^2(\N^*,\R)$, there exists $\mu \geq 0$ such that $A(f) = \mu A(g)$, that is, $A(f-\mu g) = 0$ and $\langle f - \mu g, e_k \rangle_{X^* \times X} = 0$ for every $k \in \N^*$. From density and homogeneity, we easily obtain that $\langle f - \mu g, x \rangle_{X^* \times X} = 0$ for every $x \in X$, \textit{i.e.} $f-\mu g=0$. \\

\textit{Third step ($N^*$ is weak* lower semicontinuous on $(X^* ,\Vert \cdot \Vert_*)$).} Let $(f_n) \subset X^*$ and $f \in X^*$ such that $(f_n)$ weak* converges to $f$ in $(X^* ,\Vert \cdot \Vert_*) = \mathrm{dual} ((X,\Vert \cdot \Vert))$, \textit{i.e.} $\langle f_n , x \rangle_{X^* \times X}$ tends to $\langle f , x \rangle_{X^* \times X}$ for every $x \in X$ (see \cite[Proposition 3.13 p.63]{brez}). Let us prove that $N^* (f) \leq \liminf N^*(f_n)$. First of all, recall that $(\Vert f_n \Vert_*)$ is bounded by some $M \geq 0$ and that $\Vert f \Vert_* \leq \liminf \Vert f_n \Vert_*$ (see \cite[Proposition 3.13 p.63]{brez}). Thus, $\liminf N^* (f_n) \geq \liminf \Vert f_n \Vert_* + \liminf \Vert A( f_n ) \Vert_{\ell^2} \geq \Vert f \Vert_* + \liminf \Vert A( f_n ) \Vert_{\ell^2}$. Let us recall that $\Vert A( f_n ) \Vert_{\ell^2}^2 = \sum_{k \in \N^*} \frac{1}{2^k} \langle f_n , e_k \rangle^2_{X^* \times X}$. Since $(\frac{1}{2^k} \langle f_n , e_k \rangle^2_{X^* \times X})$ converge to $\frac{1}{2^k} \langle f , e_k \rangle^2_{X^* \times X}$ and since $\vert \frac{1}{2^k} \langle f_n , e_k \rangle^2_{X^* \times X} \vert \leq \frac{1}{2^k} \Vert f_n \Vert_*^2 \Vert e_k \Vert^2 \leq \frac{M^2}{2^k} \in \ell^1(\N^* , \R)$ for every $n \in \N$, we obtain from the classical Lebesgue dominated convergence theorem that $\lim \Vert A(f_n) \Vert_{\ell^2} = \Vert A(f) \Vert_{\ell^2}$. Finally, we have proved that $\liminf N^* (f_n) \geq \Vert f \Vert_* + \Vert A( f ) \Vert_{\ell^2} = N^*(f) $. \\

\textit{Fourth step (conclusion).} We conclude the proof by applying Lemma~\ref{lem69876}.
 \end{proof}

\subsection{The distance function}\label{app2}
For this section, we essentially refer to \cite{mordu1,penot}. \\

In the whole section, $(X,\Vert \cdot \Vert)$ denotes a normed linear space and $(X^* ,\Vert \cdot \Vert_*) := \mathrm{dual} ( (X,\Vert \cdot \Vert ))$. In this section, we denote by $\langle \cdot , \cdot \rangle := \langle \cdot , \cdot \rangle_{X^* \times X}$.

\begin{lemma}\label{lemadhintconvex}
Let $\S \subset X$ be a convex subset such that $\mathrm{Int}(\S) \neq \emptyset$. Then, $\mathrm{Adh}( \mathrm{Int}(\S)) = \mathrm{Adh} (\S)$.
\end{lemma}

\begin{proof}
Let us prove that $\mathrm{Adh}(\S) \subset \mathrm{Adh}( \mathrm{Int}(\S))$. Let $x \in \mathrm{Adh}(\S)$ and let $y \in \mathrm{Int}(\S)$ such that $y \neq x$. There exists $\varepsilon > 0$ such that $y +\B_X (0,\varepsilon) \subset \S$. Let us prove that the line segment $(x,y) \subset \mathrm{Int}(\S)$. Let $w := t x + (1-t) y$ for some $t \in (0,1)$. Since $x \in \mathrm{Adh}(\S)$, there exists $x' \in \S$ such that $\Vert x-x' \Vert < \varepsilon \frac{1-t}{t}$. Let us denote by $z := x-x'$. Finally, we have $w = t x' + (1-t) y + t  z = t x' + (1-t) (y+\xi)$ where $\xi := \frac{t}{1-t}z \in \B_X (0,\varepsilon)$. As a consequence, since $x' \in \S$, since $y + \B_X (0,\varepsilon) \subset \S$ and since $\S$ is convex, $w \in t x' + (1-t) (y +  \B_X (0,\varepsilon) ) \subset \S$, that is, $w$ belongs to an open subset included in $\S$, \textit{i.e.} $w \in \mathrm{Int}(\S)$. Finally, we have proved that $(x,y) \subset \mathrm{Int}(\S)$ and thus $x \in \mathrm{Adh}(\mathrm{Int}(\S))$.
\end{proof}

Let $\varphi : X \to \R$ be a convex function and $x \in X$. Recall that the subdifferential of $\varphi$ at $x$ is defined by
$$ \partial \varphi (x) := \{ f \in X^* \; \mid \; \langle f , y - x \rangle \leq \varphi (y) - \varphi (x), \; \forall y \in X \}.$$
Note that $\partial \varphi (x)$ is a weakly* closed convex subset of $X^*$.

\begin{lemma}\label{lem95476}
If $\varphi : X \to \R$ is continuous, then $\partial \varphi (x) \neq \emptyset$ for every $x \in X$.
\end{lemma}

\begin{proof}
Recall that the epigraph of $\varphi$ is defined by
$$ \mathrm{Epi}_\varphi := \{ (y,\lambda) \in X \times \R \; \mid \; \varphi (y) \leq \lambda  \} \subset X \times \R. $$
Since $\varphi$ is convex and continuous on $X$, $\mathrm{Epi}_\varphi$ is clearly a nonempty closed convex subset of $X \times \R$. Let $x \in X$. It is clear that $(x,\varphi (x)) \notin \mathrm{Int}(\mathrm{Epi}_\varphi)$ (since $(x,\varphi (x)-\varepsilon) \notin \mathrm{Epi}_\varphi$ for all $\varepsilon > 0$) and that $\mathrm{Int}(\mathrm{Epi}_\varphi)$ is a nonempty\footnote{$\mathrm{Int}(\mathrm{Epi}_\varphi)$ is nonempty since, from Lemma~\ref{lemadhintconvex}, $\mathrm{Adh}(\mathrm{Int} (\mathrm{Epi}_\varphi)) = \mathrm{Adh}(\mathrm{Epi}_\varphi) = \mathrm{Epi}_\varphi$ is nonempty.} open convex\footnote{One can easily prove that the interior of a convex set is a convex set.} subset of $X \times \R$. From the classical Hahn-Banach theorem (see \cite[Lemma 1.3 p.6]{brez}), there exists $(f,c) \in (X \times \R)^* = X^* \times \R$ such that $\langle (f,c) , (x,\varphi(x)) \rangle < \langle (f,c) , (w,r) \rangle$, \textit{i.e.} $\langle f , w-x \rangle +c (r - \varphi (x)) > 0$ for all $(w,r) \in \mathrm{Int} (\mathrm{Epi}_\varphi)$. From the continuity of $\varphi$ at $x$, it follows that $(x,\varphi (x) + 1) \in \mathrm{Int} (\mathrm{Epi}_\varphi)$. Thus, taking $(w,r) = (x, \varphi (x)+1)$, we obtain that $c > 0$. From Lemma~\ref{lemadhintconvex}, we also obtain that $\langle f , w-x \rangle +c (r - \varphi (x)) \geq 0$ for all $(w,r) \in \mathrm{Epi}_\varphi$. Finally, taking $r = \varphi (w)$, we conclude that $ \langle - \frac{1}{c} f , w-x \rangle \leq \varphi (w) - \varphi (x)$ for all $w \in X$, \textit{i.e.} $-\frac{1}{c} f \in \partial \varphi (x)$. The proof is complete.
\end{proof}

\begin{lemma}\label{lem54891}
If $\varphi : X \to \R$ is $L$-Lipschitz continuous on $X$ for some $L \geq 0$, then $\partial \varphi (x) \neq \emptyset$ and $\partial \varphi (x) \subset \overline{\B}_{X^*} (0,L)$ for every $x \in X$.
\end{lemma}

\begin{proof}
Let $x \in X$ and $f \in \partial \varphi (x) \neq \emptyset$ (see Lemma~\ref{lem95476}). One can easily obtain that $\vert \langle f , z \rangle \vert = \langle f , \pm z \rangle = \langle f , \pm z + x - x \rangle \leq \varphi (\pm z+x) - \varphi (x) \leq \vert \varphi (\pm z+x) - \varphi (x) \vert \leq  L \Vert \pm z + x - x \Vert = L \Vert z \Vert $ for every $z \in X$. Thus $\Vert f \Vert_* \leq L$.
\end{proof}

\begin{lemma}
If $\varphi : X \to \R$ is G\^ateaux-differentiable at some $x \in X$, then $\partial \varphi (x) = \{ D\varphi (x) \}$.
\end{lemma}

\begin{proof}
Firstly, let us prove that $D\varphi (x) \in \partial \varphi (x)$. From the convexity of $\varphi$, it holds that $\varphi (x+\lambda (y-x)) \leq \varphi (x) + \lambda (\varphi (y) - \varphi (x))$, \textit{i.e.} $\frac{\varphi (x+\lambda (y-x)) - \varphi (x)}{\lambda} \leq \varphi (y) - \varphi (x)$ for any $y \in X$ and any $\lambda \in (0,1)$. Passing to the limit $\lambda \to 0^+$, we obtain that $\langle D\varphi (x) , y- x \rangle \leq \varphi(y) - \varphi (x)$ for every $y \in X$. Thus $D\varphi (x) \in \partial \varphi (x)$. Secondly, let us prove that $\partial \varphi (x) \subset \{ D\varphi (x) \}$. Let $f \in \partial \varphi (x)$. It holds that $\langle f , \lambda (y-x) \rangle \leq \varphi (x + \lambda (y-x)) - \varphi (x)$, \textit{i.e.} $\langle f , y-x \rangle \leq \frac{\varphi (x + \lambda (y-x)) - \varphi (x)}{\lambda}$ for every $y \in X$ and $\lambda \in (0,1)$. Passing to the limit $ \lambda \to 0^+$, we get that $\langle f , y-x \rangle \leq \langle D\varphi(x) ,y -x \rangle$ for every $y \in X$. From linearity, we conclude that $\langle f , y \rangle = \langle D\varphi (x) ,y \rangle$ for all $y \in X$, \textit{i.e.} $f = D\varphi (x)$.
\end{proof}

\begin{lemma}\label{lemmordu1}
If $\varphi : X \to \R$ is Lipschitz continuous around $x \in X$ and if $\partial \varphi (x)$ is reduced to a singleton $\{ f \}$, then $\varphi$ is strictly Hadamard-differentiable at $x$ with $D\varphi (x) = f$.
\end{lemma}

\begin{proof}
See \cite[Theorem 3.54 p.313]{mordu1}.
\end{proof}

\begin{lemma}\label{lem12384}
Let $\S \subset X$ be a nonempty subset and let $d_\S : X \to \R$ be the distance function to $\S$. The following properties hold:
\begin{enumerate}
\item $d_\S$ is $1$-Lipschitz continuous;
\item If $\S$ is convex, $d_\S$ is convex;
\item If $\S$ is closed and convex, $\partial d_\S (x) \neq \emptyset$ and $\partial d_\S (x) \subset \mathrm{Sph}_{X^*} (0,1)$ for every $x \in X \setminus \S$.
\end{enumerate}
\end{lemma}

\begin{proof}
$1.$ Let $x$, $y \in \S$. Let $(y_n) \subset \S$ be a sequence such that $\Vert y - y_n \Vert \to d_\S (y)$. Then it holds that $d_\S(x) \leq \Vert x-y_n \Vert \leq \Vert x-y \Vert + \Vert y - y_n \Vert \to \Vert x-y \Vert + d_\S (y)$. In a very similar way we obtain that $d_\S (y) \leq \Vert x-y \Vert +d_\S (x)$. We conclude that $\vert d_\S (y) - d_\S (x) \vert \leq \Vert y-x \Vert$. \\

$2.$ Let $x$, $y \in \S$ and $\lambda \in [0,1]$. Let $(x_n)$, $(y_n) \subset \S$ such that $\Vert x - x_n \Vert \to d_\S (x)$ and $\Vert y - y_n \Vert \to d_\S (y)$. Note that $(1-\lambda)x_n + \lambda y_n \in \S$ since $\S$ is convex. Then, $d_\S ((1-\lambda)x + \lambda y ) \leq \Vert [(1-\lambda)x+\lambda y] - [(1-\lambda)x_n+\lambda y_n ] \Vert \leq (1-\lambda) \Vert x - x_n \Vert + \lambda \Vert y- y_n \Vert \to (1-\lambda ) d_\S (x) + \lambda d_\S (y)$. Thus, $d_\S$ is convex. \\

$3.$ Let $x \in X \setminus \S$ and $f \in \partial d_\S (x) \neq \emptyset$ (see Lemma~\ref{lem54891}). From Lemma~\ref{lem54891}, we already know that $\Vert f \Vert_* \leq 1$. Since $\S$ is closed, note that $d_\S (x) > 0$. There exists $(x_n ) \subset \S$ such that $\Vert x - x_n \Vert \to d_\S (x) $. Moreover, it holds that $\langle f , x_n - x \rangle \leq d_\S (x_n) - d_\S (x) = - d_\S(x)$. Thus, $d_\S (x) \leq \langle f , x-x_n \rangle \leq \Vert f \Vert_* \Vert x - x_n \Vert \to \Vert f \Vert_* d_\S (x)$. Thus, $\Vert f \Vert_* \geq 1$. We conclude that $\partial d_\S (x) \subset \mathrm{Sph}_{X^*} (0,1)$.
\end{proof}

\begin{remark}\label{remarkdistance}
Since $d_\S$ is $1$-Lipschitz continuous on $X$, one can easily prove that $d_\S^2$ is Fr\'echet-differentiable on $\S$, with $Dd^2_\S (x) = 0$ for every $x \in \S$.
\end{remark}

\begin{proposition}\label{propdistance}
Let us assume that $\Vert \cdot \Vert_*$ is strictly convex on $X^*$. Let $\S \subset X$ be a nonempty closed and convex subset. Then, $d_\S$ is strictly Hadamard-differentiable on $X \setminus \S$ with $ \Vert Dd_\S (x) \Vert_* = 1$ for every $x \in X \setminus \S$.
\end{proposition}

\begin{proof}
Let $x \in X \setminus \S$. From Lemma~\ref{lem12384}, it follows that $\partial d_\S (x) \neq \emptyset$ and $\partial d_\S (x) \subset \mathrm{Sph}_{X^*} (0,1)$. Since $\partial d_\S (x) \neq \emptyset$ is a convex subset included in $\mathrm{Sph}_{X^*} (0,1)$ and since $\Vert \cdot \Vert_*$ is strictly convex, it follows that $\partial d_\S (x)$ is necessarily reduced to a singleton $\{ f \}$ satisfying $\Vert f \Vert_* = 1$. We conclude from Lemma~\ref{lemmordu1}.
\end{proof}

\section{Functions of bounded variations and Stieltjes integrals}\label{annexeBV}

For this section, we essentially refer to \cite{bach,burk,conw,fara,lima,whee}.

\subsection{Some recalls}\label{annexeBV1}
Recall that a function $\eta : [0,T] \to \R$ is said to be \textit{of bounded variations} if
$$ \V (\eta) := \sup_{(t_k)_k} \Big\{ \sum_{k} \vert \eta (t_{k+1}) - \eta (t_k) \vert \Big\} < +\infty  $$
where the supremum is taken over all partitions $(t_k)_k$ of $[0,T]$. In that case, we denote by $\eta \in \BV_1$. \\

Recall that $\eta : [0,T] \to \R$ belongs to $\BV_1$ if and only if $\eta$ is equal to the difference of two monotically increasing functions. In particular, if $\eta \in \BV_1$, then $\eta$ admits a left-limit denoted by $\eta (t^-)$ (resp. a right-limit denoted by $\eta(t^+)$) at every point $t \in (0,T]$ (resp. $t \in [0,T)$) and the set of discontinuity points of $\eta$ is at most countable. \\ 

Recall that the application $\Vert \cdot \Vert_{\BV_1}$ given by
$$ \fonction{\Vert \cdot \Vert_{\BV_1}}{\BV_1}{\R^+}{\eta}{\eta (0) + \V (\eta)} $$
defines a norm on $\BV_1$. Moreover, recall that $(\BV_1,\Vert \cdot \Vert_{\BV_1})$ is a Banach space. \\

Let $\eta \in \BV_1$. For every $t \in [0,T]$, $\vert \eta (t) \vert \leq \vert \eta(0) \vert + \vert \eta (t) -\eta(0)\vert \leq  \vert \eta(0) \vert + \vert \eta (t) -\eta(0)\vert + \vert \eta (T) -\eta(t) \vert \leq \vert \eta (0) \vert + \V(\eta) \leq \Vert \eta \Vert_{\BV_1}$. As a consequence, $\eta \in \BF_1$ with $\Vert \eta \Vert_\infty \leq \Vert \eta \Vert_{\BV_1}$. \\

Recall that if $z \in \C_1$ and $\eta \in \BV_1$, then the classical Riemann-Stieltjes integral defined by
$$ \int_0^T z(\t) \; d\eta(\t) := \lim \sum_k z(t_k) (\eta(t_{k+1})-\eta(t_k) ) $$ 
exists. In the above equality, the limit means that the length of the partition $(t_k)_k$ tends to zero. \\

Let us recall the following classical Riesz theorem.

\begin{proposition}[Riesz theorem]\label{propriesz}
Let $\varphi \in \C^*_1$. There exists a unique $\eta \in \NBV_1$ such that
$$ \langle \varphi , z \rangle_{\C^*_1 \times \C_1} = \int_0^T z(\t) \; d\eta (\t), $$
for every $z \in \C_1$. Moreover, if $\langle \varphi , z \rangle_{\C^*_1 \times \C_1} \geq 0$ for every $z \in \C^+_1$, then $\eta$ is monotically increasing.
\end{proposition}

We refer to \cite[p.245]{lima} for a complete proof of Proposition~\ref{propriesz}. In this section we will only detail the proof of the following weaker result, which is sufficient for the completeness of this note. 

\begin{proposition}[Riesz corollary]\label{propriesz2}
Let $\varphi \in \C^*_1$ such that $\langle \varphi , z \rangle_{\C^*_1 \times \C_1} \geq 0$ for every $z \in \C^+_1$. There exists $\eta \in \NBV_1$ such that $\eta $ is monotically increasing on $[0,T]$ and
$$ \langle \varphi , z \rangle_{\C^*_1 \times \C_1} = \int_0^T z(\t) \; d\eta (\t), $$
for every $z \in \C_1$. Moreover, from construction of $\eta$, $\varphi = 0$ if and only if $\eta = 0$.
\end{proposition}

\begin{proof}
If $\varphi = 0$, it is sufficient to consider $\eta = 0 \in \NBV_1$ that is monotically increasing. Now let us consider that $\varphi \neq 0$. In the sequel we simply denote by $1 \in \C^+_1 \subset \BF^+_1$ the constant function equal to $1$ on $[0,T]$. Finally we denote by $\langle \cdot , \cdot \rangle := \langle \cdot , \cdot \rangle_{\C^*_1 \times \C_1}$. When no confusion is possible, $\langle \cdot , \cdot \rangle$ also denotes $\langle \cdot , \cdot \rangle_{\BF^*_1 \times \BF_1}$.  \\

Since $\langle \varphi , z \rangle  \geq 0$ for every $z \in \C^+_1$, it holds that $1-\xi \in \C_1^+$ and then $\langle \varphi , \xi \rangle \leq \langle \varphi , 1 \rangle$ for every $\xi \in \C_1$ such that $\Vert \xi \Vert_\infty \leq 1$. We deduce that $\Vert \varphi \Vert_{\C^*_1} \leq \langle \varphi , 1 \rangle \leq \Vert \varphi \Vert_{\C^*_1} \Vert 1 \Vert_\infty = \Vert \varphi \Vert_{\C^*_1}$ and then $\Vert \varphi \Vert_{\C^*_1} = \langle \varphi , 1 \rangle$. From the classical Hahn-Banach theorem (see \cite[Corollary 1.2 p.3]{brez}), there exists a linear continuous application $\tilde{\varphi} : \BF_1 \to \R$ such that $\tilde{\varphi}$ extends $\varphi$ to $\BF_1$ and $\Vert \tilde{\varphi} \Vert_{\BF_1^*} = \Vert \varphi \Vert_{\C^*_1}$. In particular it holds that $\Vert \tilde{\varphi} \Vert_{\BF_1^*} = \langle \tilde{\varphi} , 1 \rangle$. Let us prove that $\langle \tilde{\varphi} , z \rangle  \geq 0$ for every $z \in \BF^+_1$. Let $z \in \BF^+_1$ such that $z \neq 0$ and let $\xi := \frac{2}{\Vert z \Vert_\infty} z - 1$. Then $\xi \in \BF_1$ with $\Vert \xi \Vert_\infty \leq 1$. As a consequence $- \langle \tilde{\varphi} , \xi \rangle \leq \vert \langle \tilde{\varphi} , \xi \rangle \vert \leq \Vert \tilde{\varphi} \Vert_{\BF_1^*} \Vert \xi \Vert_\infty \leq \langle \tilde{\varphi} , 1 \rangle$. Finally, it holds that $\langle \tilde{\varphi} , z \rangle = \frac{ \Vert z \Vert_\infty }{2} (\langle \tilde{ \varphi} , \xi \rangle + \langle \varphi , 1 \rangle ) \geq 0$. \\

Now we introduce $ \eta (t) := \langle \tilde{\varphi} , \mathbf{1}_{(0,t]} \rangle $ for every $t \in [0,T]$. In particular $\eta (0 ) = 0$. Moreover, since $\mathbf{1}_{(0,t]} - \mathbf{1}_{(0,s]} \in \BF^+_1$ for every $0 \leq s \leq t \leq T$ and since $\langle \tilde{\varphi} , z \rangle  \geq 0$ for every $z \in \BF^+_1$, it follows that $\eta$ is monotically increasing on $[0,T]$. In particular $\eta \in \BV_1$ with $\V (\eta) = \eta (T) - \eta (0)$. \\

Now let us prove that $ \langle \varphi , z \rangle = \int_0^T z(\t) d\eta (\t) $ for every $z \in \C_1$. Let $z \in \C_1$ and let $\varepsilon > 0$. Since $z$ is uniformly continuous on $[0,T]$, there exists $\delta > 0$ such that
$$ \Vert \varphi \Vert_{\C^*_1} \vert z (t) - z(s) \vert \leq \frac{\varepsilon}{2} $$
for every $(t,s) \in [0,T]^2$ such that $\vert t - s \vert \leq \delta$. Let $(t_k)_k$ be a partition of $[0,T]$ such that $t_{k+1} - t_k \leq \delta$ and such that
$$ \left\vert \int_0^T z(\t) \; d\eta (\t) - \sum_k z(t_k) (\eta (t_{k+1})- \eta (t_k) ) \right\vert \leq \frac{\varepsilon}{2}. $$
Then, we introduce $ u = \sum_k z(t_k) \mathbf{1}_{(t_k,t_{k+1}]} \in \BF_1$. It clearly holds that
$$ \Vert \varphi \Vert_{\C^*_1} \Vert z - u \Vert_\infty  \leq \frac{\varepsilon}{2} \quad \text{and} \quad \langle \tilde{\varphi} , u \rangle = \sum_k z(t_k) (\eta (t_{k+1})- \eta (t_k) ). $$
As a consequence it holds that
$$ \left\vert \int_0^T z(\t) \; d\eta (\t) - \langle \varphi , z \rangle \right\vert \leq  \left\vert \int_0^T z(\t) \; d\eta (\t) - \langle  \tilde{\varphi}  , u \rangle \right\vert + \vert \langle \tilde{\varphi} , u \rangle - \langle \varphi , z \rangle \vert \leq {\varepsilon} . $$
In particular, it holds that $\eta (T) = \eta (T) - \eta (0) = \int_0^T 1 d\eta (\t) = \langle \varphi , 1 \rangle = \Vert \varphi \Vert_{\C_1^*} \neq 0$. Thus $\eta \neq 0$. \\

To conclude the proof, let us prove that $\eta$ can be chosen left-continuous on $(0,T)$. First of all, since $\eta (0) = 0$ and $\eta $ is monotically increasing on $[0,T]$, $\eta (t^-)$ exists for all $0 < t \leq T$ and it holds that
\begin{equation}\label{eq980a0}
0 \leq \eta (s) \leq \eta (t^-) \leq \eta (t) \leq \eta (T),
\end{equation}
for every $0 \leq s < t \leq T$. Let us define
$$ \nu (t) := \left\lbrace \begin{array}{lcl}
0 & \text{if} & t=0 , \\
\eta (t^-) & \text{if} & 0 < t < T , \\
\eta (T) & \text{if} & t=T.
\end{array} \right. $$
In particular $\nu \neq 0$. Using \eqref{eq980a0}, one can easily prove that $\nu$ is monotically increasing on $[0,T]$. Let us assume by contradiction that $\nu $ is not left-continuous on $(0,T)$, that is, there exists $t \in (0,T)$ and $\varepsilon > 0$ such that $\nu (t^-) \leq \nu (t) - \varepsilon$. Finally, from \eqref{eq980a0}, we obtain that 
$$ \eta \left( t - \frac{2}{k} \right) = \eta \left(t - \frac{1}{k} - \frac{1}{k} \right) \leq \eta \left( \left(t- \frac{1}{k} \right)^- \right) = \nu \left( t - \frac{1}{k} \right) \leq \nu (t^- ) \leq \nu (t) - \varepsilon = \eta (t^- ) - \varepsilon, $$
for every $k \in \N^*$. This raises a contradiction when $k$ tends to $+\infty$. Hence $\nu \in \NBV_1$. Now let us prove that
$$ \int_0^T z(\t) \; d\eta (\t) = \int_0^T z(\t) \; d\nu (\t) $$
for every $z \in \C_1$. Since $\eta$ and $\nu$ can be different only at discontinuity points of $\eta$ (which are at most countable), we consider a sequence of partitions $((t^\ell_k)_k)_\ell$ of $[0,T]$ such that the length of the partitions tends to zero when $\ell$ tends to $+\infty$ and such that no point $t^\ell_k$ is a discontinuity point of $\eta$. As a consequence, it holds that
$$ \int_0^T z(\t) \; d\eta (\t) = \lim\limits_{\ell \to \infty} \sum_k z(t_k) ( \eta (t_{k+1}) - \eta (t_k) ) =  \lim\limits_{\ell \to \infty} \sum_k z(t_k) ( \nu (t_{k+1}) - \nu (t_k) ) = \int_0^T z(\t) \; d\nu  (\t) $$
for every $z \in \C_1$.
\end{proof}

In Appendix~\ref{appderniereCSP}, we will need the two following results.

\begin{lemma}\label{lemACBV}
If $h \in \AC_1$, then $h \in \BV_1$.
\end{lemma}

\begin{proof}
One can easily get that
$$ \sum_k \vert h(t_{k+1}) - h(t_k) \vert = \sum_k \left\vert \int_{t_k}^{t_{k+1}} \dot{h}(\t) \; d\t \right\vert \leq \int_0^T \vert \dot{h}(\t) \vert \; d\t = \Vert \dot{h} \Vert_{\L^1}, $$
for every partition $(t_k)_k$ of $[0,T]$.
\end{proof}

\begin{lemma}\label{lemprimitiveBV}
Let $\eta \in \NBV_1$ be monotically increasing and $z \in \C_1$. The function $h : [0,T] \to \R$ defined by
$$ h(t) := \int_0^t z(\t) \; d\eta(\t), $$
is of bounded variations, \textit{i.e.} $h \in \BV_1$.
\end{lemma}

\begin{proof}
One can easily get that 
$$ \sum_k \vert h(t_{k+1}) - h(t_k) \vert = \sum_k \left\vert \int_{t_k}^{t_{k+1}} z(\t) \; d\eta(\t) \right\vert \leq \sum_k \int_{t_k}^{t_{k+1}} \vert z(\t) \vert \; d\eta(\t) \leq \int_{0}^{T} \vert z(\t) \vert \; d\eta(\t), $$
for every partition $(t_k)_k$ of $[0,T]$.
\end{proof}

We conclude this section with some recalls about the Lebesgue-Stieltjes integral. Let $\eta \in \NBV_1$ be monotically increasing. Then $\eta $ induces a finite nonnegative measure on the Borel set of $[0,T]$ denoted by $d \eta$. This measure is constructed from the equalities $d\eta ( [a,b) ) = \eta (b) - \eta (a)$ for every $0 \leq a \leq b \leq T$ and extended from the classical Carath\'eodory extension theorem. For every $z \in \C_1$, the Riemann-Stieltjes integral of $z$ with respect to $\eta$ and the Lebesgue-Stieltjes integral of $z$ with respect to $d\eta$ (that corresponds to the classical Lebesgue integral of $z$ with respect to the measure $d\eta$) coincide. We refer to \cite[p.83]{fara} or \cite[p.288]{whee} for more details. Finally, the following Fubini-type formula holds:
\begin{equation}\label{eqfubini1D}
\int_0^T \int_0^\t z(\t,s) \; ds \; d\eta (\t) = \int_0^T \int_s^T z(\t,s) \; d\eta(\t) \; ds
\end{equation}
for every $z \in \L^\infty([0,T]^2,\R)$ such that $z$ is continuous in its first variable. 

\begin{remark}
Note that the Stieltjes integrals in \eqref{eqfubini1D} are both well-defined in terms of Riemann-Stieltjes integration\footnote{Indeed, one can easily prove from the classical Lebesgue dominated convergence theorem that the function $\t \mapsto \int_0^\t z (\t,s) \; ds$ is continuous on $[0,T]$.} and the classical integrals in \eqref{eqfubini1D} have to be understood in the Lebesgue sense. Actually, one can easily see that the double integrals in \eqref{eqfubini1D} both exist.
\end{remark}

\subsection{Notations and Fubini-type formulas}\label{annexeBVnotations}
For any $\eta = (\eta_i)_{i=1,\ldots,j} \in \NBV_j$ such that $\eta_i$ is monotically increasing and for any $z = (z_i)_{i=1,\ldots,j} \in \C_j$, we denote by
$$ \int_0^T \langle z(\t) , d\eta (\t) \rangle := \sum_{i=1}^j \int_0^T z_i (\t) \; d\eta_i (\t) \in \R. $$
Let $r \in \N^*$. We denote by
$$ \int_0^T A(\t) \times d\eta (\t) := \left( \sum_{i=1}^j \int_0^T a_{ki}(\t) \; d\eta_i (\t) \right)_{k=1,\ldots,r} \in \R^r, $$
and
$$ \int_0^T \langle z(\t) , A(\t) \times d\eta (\t) \rangle := \int_0^T \langle A(\t)^\top \times z(\t) , d\eta (\t) \rangle \in \R, $$
for every continuous matrices $A(\cdot) =(a_{ki} (\cdot))_{ki} : [0,T] \to \R^{r,j}$ and every $z \in \C_r$. Moreover, one can easily prove that if $z \in \R^r$ (\textit{i.e.} $z \in \C_r$ constant) then
\begin{equation}\label{eq7645}
\int_0^T \langle z , A(\t) \times d\eta (\t) \rangle = \left\langle z , \int_0^T A(\t) \times d\eta (\t) \right\rangle_{\R^r \times \R^r}.
\end{equation}
Finally, one can prove from Equality~\eqref{eqfubini1D} that the following Fubini-type formulas both hold:
\begin{equation}\label{eq326}
\int_0^T \left\langle \int_0^\t \Phi (\t,s) \; ds , d\eta (\t) \right\rangle = \int_0^T \int_s^T \langle \Phi (\t,s) , d\eta (\t) \rangle ds,
\end{equation}
and
\begin{equation}\label{eq326-2}
\int_0^T \left( \int_0^\t A (\t,s) \; ds \right) \times d\eta (\t) = \int_0^T \left( \int_s^T A (\t,s) \times d\eta (\t) \right)  ds,
\end{equation}
where $ \Phi \in \L^\infty([0,T]^2,\R^j) $ and $A \in \L^\infty([0,T]^2,\R^{r,j})$ are continuous in their first variable. 

\section{State-transition matrices and linear Cauchy-Stieltjes problems}\label{appderniere}
In the whole section $A \in \L^\infty([0,T],\R^{n, n})$.

\subsection{Recalls on state-transition matrices}\label{appstatetransition}
For every $s \in [0,T]$, the backward/forward linear Cauchy problem~\eqref{eqLBFCPA} given by
\begin{equation}\label{eqLBFCPA}\tag{BFCP${}_{A,s}$}
 \left\lbrace \begin{array}{l}
\dot{Z}(t) = A(t) \times Z(t), \quad \text{a.e. $t \in [0,T]$,}  \\[5pt]
Z(s)=\mathrm{Id}_n,
\end{array} \right. 
\end{equation}
admits a unique maximal solution that is moreover global.\footnote{This results follows from the classical linear version of the Cauchy-Lipschitz (or Picard-Lindel\"of) theorem.} We denote this solution by $Z(\cdot,s) : [0,T] \to \R^{n,n}$. The matrix function $Z(\cdot,\cdot)$ is the so-called \textit{state-transition matrix} associated to $A$. 

\begin{lemma}\label{lemstatetransi}
The following equalities both hold
\begin{eqnarray*}
Z(t,s) & = & \mathrm{Id}_n + \int_s^t A(\t) \times Z(\t,s) \; d\t, \\
& = & \mathrm{Id}_n + \int_s^t Z(t,\t) \times A(\t) \; d\t,
\end{eqnarray*}
for every $(t,s) \in [0,T]^2$. In particular, $Z(\cdot,\cdot) : [0,T]^2 \to \R^{n,n}$ is continuous.
\end{lemma}

\begin{proof}
The first equality is obvious since it corresponds to the definition of a global solution of~\eqref{eqLBFCPA}. From this equality and from the classical Gronwall lemma, one can easily prove that $Z(\cdot,\cdot)$ is bounded on $[0,T]^2$. For every $(t,s) \in [0,T]^2$, we introduce 
$$ T(t,s) := \mathrm{Id}_n + \int_s^t Z(t,\t) \times A(\t) \; d\t , $$ 
that is well-defined since $Z(\cdot,\cdot)$ is bounded on $[0,T]^2$. Our aim is to prove that $Z(t,s) = T(t,s)$. From the first equality, it holds that
$$ \int_s^t Z(t,\t) \times A(\t) \; d\t = \int_s^t A(\t) \; d\t + \int_s^t \int_\t^t A(\xi) \times Z(\xi,\t) \times A(\t) \; d\xi \; d\t, $$
for every $(t,s) \in [0,T]^2$. Using the classical Fubini formula (and inversing the roles of $\t$ and $\xi$), we obtain that
$$ \int_s^t Z(t,\t) \times A(\t) \; d\t = \int_s^t A(\t) \left[ \mathrm{Id}_n + \int_s^\t Z(\t,\xi) \times A(\xi) \; d\xi \right] d\t . $$
Finally, adding $\mathrm{Id}_n$ in the above equality, we obtain that $T(\cdot,\cdot)$ satisfies
$$ T(t,s) = \mathrm{Id}_n + \int_s^t A(\t) \times T(\t,s) \; d\t , $$
for every $(t,s) \in [0,T]^2$. From uniqueness of the global solution of~\eqref{eqLBFCPA}, we obtain that $T(t,s) = Z(t,s)$. To conclude, from the definition of $Z(\cdot,\cdot)$, it is clear that $Z(\cdot,\cdot)$ is (absolutely) continuous in its first variable. Using the second equality and the classical Lebesgue dominated convergence theorem, one can easily prove that $Z(\cdot,\cdot)$ is continuous on $[0,T]^2$.
\end{proof}

\begin{remark}
From Lemma~\ref{lemstatetransi}, note that $Z(t,\cdot)$ is the unique global solution of the backward/forward linear Cauchy problem given by
$$ \left\lbrace \begin{array}{l}
\dot{Z}(s) = - Z(s) \times A(s), \quad \text{a.e. $s \in [0,T]$,}  \\[5pt]
Z(t)=\mathrm{Id}_n,
\end{array} \right. $$
for every $t \in [0,T]$.
\end{remark}

\subsection{Recalls on linear Cauchy problems}\label{appcauchyproblemclassique}
Let $B \in \L^\infty([0,T],\R^{n})$ and $q_0$, $p_T \in \R^n$. From the classical linear version of the Cauchy-Lipschitz (or Picard-Lindel\"of) theorem, the forward linear Cauchy problem~\eqref{eqLFCPAB} given by
\begin{equation}\label{eqLFCPAB}\tag{FCP${}_{A,B}$}
\left\lbrace \begin{array}{l}
\dot q(t) = A(t) \times q(t) + B(t), \quad \text{a.e. $t \in [0,T]$,}  \\[5pt]
q(0)=q_0,
\end{array} \right. 
\end{equation}
admits a unique maximal solution that is moreover global. Similarly, the backward linear Cauchy problem~\eqref{eqLBCPAB} given by
\begin{equation}\label{eqLBCPAB}\tag{BCP${}_{A,B}$}
\left\lbrace \begin{array}{l}
- \dot p(t) = A(t)^\top \times p(t) + B(t), \quad \text{a.e. $t \in [0,T]$,}  \\[5pt]
p(T)=p_T,
\end{array} \right. 
\end{equation}
also admits a unique maximal solution that is moreover global. 

\begin{proposition}[Duhamel formulas]\label{propduhamelclassique}
The global solutions of \eqref{eqLFCPAB} and \eqref{eqLBCPAB} are given by
$$ q(t) =  Z(t,0) \times q_0 + \int_0^t Z(t,s) \times B(s) \; ds , $$
and
$$ p(t) = Z(T,t)^\top \times p_T + \int_t^T Z(\t,t)^\top \times B(\t) \; d\t , $$
for every $t \in [0,T]$, where $Z(\cdot,\cdot)$ is the state-transition matrix associated to $A$.
\end{proposition}

\begin{proof}
Let $q : [0,T] \to \R^n$ be defined by
$$ q(t) :=  Z(t,0) \times q_0 + \int_0^t Z(t,s) \times B(s) \; ds , $$
for every $t \in [0,T]$. Replacing the value of $Z(\cdot,\cdot)$ by the first equality given in Lemma~\ref{lemstatetransi} and using the classical Fubini formula, one can easily prove that
$$ q(t) = q_0 + \int_0^t A(\t) \times q(\t) + B(\t) \; d\t , $$
for every $t \in [0,T]$. As a consequence, $q$ is the unique global solution of~\eqref{eqLFCPAB}. Similarly, let $p : [0,T] \to \R^n$ be defined by
$$ p(t) := Z(T,t)^\top \times p_T + \int_t^T Z(\t,t)^\top \times B(\t) \; d\t , $$
for every $t \in [0,T]$. Replacing the value of $Z(\cdot,\cdot)$ by the second equality given in Lemma~\ref{lemstatetransi} and using the classical Fubini formula, one can easily prove that
$$ p(t) = p_T + \int_t^T A(\t)^\top \times p(\t) + B(\t) \; d\t , $$
for every $t \in [0,T]$. As a consequence, $p$ is the unique global solution of~\eqref{eqLBCPAB}.
\end{proof}

\subsection{Linear Cauchy-Stieltjes problems}\label{appderniereCSP}
Let $q_0$, $p_T \in \R^n$. Let $B := (B_i)_{i=1,\ldots,j}$ and $\eta := (\eta_i)_{i=1,\ldots,j}$ where $B_i \in \C_n$ and $\eta_i \in \NBV_1$ is monotically increasing for every $i=1,\ldots,j$. We say that $q$ is a \textit{global solution} of the forward linear Cauchy-Stieltjes problem~\eqref{eqLFCSPAB} given by
\begin{equation}\label{eqLFCSPAB}\tag{FCSP${}_{A,B}$}
\left\lbrace \begin{array}{l}
dq = A \times q \; dt + \sum_{i=1}^j B_i \; d\eta_i, \quad \text{on $[0,T]$,}  \\[5pt]
q(0)=q_0,
\end{array} \right. 
\end{equation}
if $q \in \BF_n$ and $q$ satisfies
$$ q(t) = q_0 + \int_0^t A(\t) \times q(\t) \; d\t + \sum_{i=1}^j \int_0^t B_i (\t) \; d\eta_i (\t),  $$
for every $t \in [0,T]$. In such a case, it follows from Lemmas~\ref{lemACBV} and \ref{lemprimitiveBV} that $q \in \BV_n$. \\

Similarly, we say that $p$ is a \textit{global solution} of the backward linear Cauchy-Stieltjes problem~\eqref{eqLBCSPAB} given by
\begin{equation}\label{eqLBCSPAB}\tag{BCSP${}_{A,B}$}
\left\lbrace \begin{array}{l}
-dp = A^\top \times p \; dt + \sum_{i=1}^j B_i \; d\eta_i, \quad \text{on $[0,T]$,}  \\[5pt]
p(T)=p_T,
\end{array} \right. 
\end{equation}
if $p \in \BF_n$ and $p$ satisfies
$$ p(t) = p_T + \int_t^T A(\t)^\top \times p(\t) \; d\t + \sum_{i=1}^j \int_t^T B_i (\t) \; d\eta_i (\t),  $$
for every $t \in [0,T]$. In such a case, it follows from Lemmas~\ref{lemACBV} and \ref{lemprimitiveBV} that $p \in \BV_n$.

\begin{proposition}\label{propexistCSP}
Problem~\eqref{eqLFCSPAB} admits a unique global solution. Problem~\eqref{eqLBCSPAB} admits a unique global solution.
\end{proposition}

\begin{proof}
In this proof, we only treat Problem~\eqref{eqLFCSPAB}. Let us consider the functional given by
$$  \fonction{\mathcal{G}}{\BF_n}{\BV_n \subset \BF_n}{q}{

\fonction{\mathcal{G}(q)}{[0,T]}{\R^n}{t}{q_0 + \int_0^t A(\t) \times q(\t) \; d\t + \sum_{i=1}^j \int_0^t B_i (\t) \; d\eta_i (\t).}

} $$
Note that $\mathcal{G}$ is well-defined from Lemmas~\ref{lemACBV} and \ref{lemprimitiveBV}. Our aim is to prove that $\mathcal{G}$ admits a unique fixed point. To do so, we will prove that $\mathcal{G}$ admits a contractive iterate. One can easily prove by induction on $k \in \N^*$ that
$$ \Vert \mathcal{G}^k(q_2)(t) -  \mathcal{G}^k(q_1) (t) \Vert_{\R^n} \leq \frac{\Vert A \Vert^k_{\L^\infty} }{(k-1)!} \int_0^t (t - \t)^{k-1} \Vert q_2(\t) - q_1(\t) \Vert_{\R^n} \; d\t ,$$
for every $q_1$, $q_2 \in \BF_n$, every $t \in [0,T]$ and every $k \in \N^*$. As a consequence, it holds that
$$ \Vert \mathcal{G}^k(q_2) -  \mathcal{G}^k(q_1) \Vert_{\infty} \leq \frac{( \Vert A \Vert_{\L^\infty} T)^k}{k!} \Vert q_2 - q_1 \Vert_{\infty} ,$$
for every $q_1$, $q_2 \in \BF_n$ and every $k \in \N^*$. Taking $k \in \N^*$ sufficiently large to get $\frac{( \Vert A \Vert_{\L^\infty} T)^k}{k!} < 1$, we obtain that $\mathcal{G}^k$ is a contractive iterate of $\mathcal{G}$. Since $(\BF_n , \Vert \cdot \Vert_\infty)$ is a Banach space, we conclude that $\mathcal{G}$ admits a unique fixed point.
\end{proof}

\begin{proposition}[Duhamel-type formulas]\label{propduhamelstieltjes}
The global solutions of \eqref{eqLFCSPAB} and \eqref{eqLBCSPAB} are given by
$$ q(t) =  Z(t,0) \times q_0 + \sum_{i=1}^j \int_0^t Z(t,s) \times B_i (s) \; d\eta_i (s) , $$
and
$$ p(t) = Z(T,t)^\top \times p_T + \sum_{i=1}^j \int_t^T Z(\t,t)^\top \times B_i(\t) \; d\eta_i(\t) , $$
for every $t \in [0,T]$, where $Z(\cdot,\cdot)$ is the state-transition matrix associated to $A$.
\end{proposition}

\begin{proof}
From the Fubini-type formulas provided in Appendix~\ref{annexeBV} for Stieltjes integrals, the proof is exactly the same than in Proposition~\ref{propduhamelclassique}.
\end{proof}


\begin{thebibliography}{10}

\bibitem{acker}
J.E. Ackermann.
\newblock {\em Sampled-data control. Volume 1}.
\newblock Springer-Verlag, Berlin-New-York, 1983.

\bibitem{bach}
G.~Bachman and L. Narici.
\newblock {\em Functional analysis}.
\newblock Academic Pr, 1966.

\bibitem{bonnans}
J.F. Bonnans and C. De La Vega.
\newblock Optimal control of state constrained integral equations.
\newblock {\em Set-Valued Analysis}, 18(3):307--326, 2010.

\bibitem{brez}
H.~Brezis.
\newblock {\em Functional analysis, Sobolev spaces and partial differential equations}.
\newblock Springer, New York, 2011.

\bibitem{burk}
F.E.~Burk.
\newblock {\em A garden of integrals}.
\newblock Mathematical Association of America, 2007.

\bibitem{conw}
J.B.~Conway.
\newblock {\em A course in abstract analysis}.
\newblock American Mathematical Society, 2012.

\bibitem{fara}
J.~Faraut.
\newblock {\em Calcul int\'egral}.
\newblock EDP Sciences, 2006.

\bibitem{fons}
I.~Fonseca and G. Leoni.
\newblock {\em Modern methods in the calculus of variations: $L^p$ spaces}.
\newblock Springer, 2007.

\bibitem{frys}
A.~Fryszkowski.
\newblock {\em Fixed point theory for decomposable sets}.
\newblock Springer Netherlands, 2004.

\bibitem{lima}
B.V. Limaye.
\newblock {\em Functional analysis}.
\newblock New Age International, 1996.

\bibitem{liyong}
X. Li and J. Yong.
\newblock {\em Optimal control theory for infinite dimensional systems}.
\newblock Birkh\"auser Boston, 1995.

\bibitem{mordu1}
B.~Mordukhovich.
\newblock {\em Variational Analysis and Generalized Differentiation I}.
\newblock Springer-Verlag, Berlin Heidelberg, 2006.

\bibitem{penot}
J.-P.~Penot.
\newblock {\em Calculus without derivatives}.
\newblock Springer-Verlag, New York, 2013.

\bibitem{sier}
W. Sierpinski.
\newblock Sur les fonctions d'ensemble additives et continues.
\newblock {\em Fundamenta Mathematicae}, 3:240--246, 1922.

\bibitem{whee}
R.-L.~Wheeden and A. Zygmund.
\newblock {\em Measure and integral: an introduction to real analysis}.
\newblock Monographs and textbooks in pure and applied mathematics, 2015.

\end{thebibliography}
\end{document}